\documentclass[10pt,a4paper,twoside,reqno]{amsart}

 \usepackage{graphicx,color,amssymb, amsmath, amsthm, mathrsfs}

\usepackage{tikz}
\usepackage{enumitem}
\usepackage{slashed}
\usepackage{amsmath,amscd}
\usepackage{amssymb,amsfonts,mathrsfs}
\usepackage[driver=pdftex,margin=3cm,heightrounded=true,centering]{geometry}

\usepackage{JK}
\usepackage[colorlinks=true,linkcolor=blue,citecolor=blue]{hyperref}

\newtheorem{thm}{Theorem}[section]

\usepackage{graphicx}

\newcommand{\A}{\mathcal{A}}
\begin{document}
\title[Operator $*$-correspondences]{Operator $*$-correspondences in analysis and geometry}

\author{David Blecher}
\address{Department of Mathematics, 
University of Houston, 
Houston, 
TX 77204-3008,
Texas}
\email{dblecher@math.uh.edu}

\author{Jens Kaad}
\address{Department of Mathematics and Computer Science, 
The University of Southern Denmark, 
Campusvej  55, 
DK-5230 Odense M,
Denmark}
\email{jenskaad@hotmail.com}

\author{Bram Mesland}
\address{Institut f\"ur Analysis, 
Leibniz Universit\"at Hannover, 
Welfengarten 1, 
30167 Hannover, 
Germany}
\email{mesland@math.uni-hannover.de}


\begin{abstract} An operator $*$-algebra is a non-selfadjoint operator algebra with completely isometric involution. We show that any operator $*$-algebra admits a faithful representation on a Hilbert space in such a way that the involution coincides with the operator adjoint up to conjugation by a symmetry. We introduce operator $*$-correspondences as a general class of inner product modules over operator $*$-algebras and prove a similar representation theorem for them. From this we derive the existence of linking operator $*$-algebras for operator $*$-correspondences. We illustrate the relevance of this class of inner product modules by providing numerous examples arising from noncommutative geometry.
\end{abstract}

\maketitle
\tableofcontents
\section*{Introduction}

There is a large literature on (possibly nonselfadjoint) algebras of
operators on a Hilbert space--{\em operator algebras}, and on
modules over such algebras. It has been known for a long time
that to effectively treat many aspects of these objects  one needs
to view them as {\em operator spaces}, that is, treat them as
{\em matrix normed} objects (see e.g.\ \cite{Arvsub,Pis:IOT,BlMe:OMO}). In (non)commutative geometry dense subalgebras of $C^{*}$-algebras, such as differentiable functions on a manifold, arise naturally. Such algebras commonly carry a finer topology that is compatible with the ambient $C^{*}$-algebra, \cite{Con:NCG}.

In recent years, considering such dense $*$-subalgebras as non-selfadjoint operator algebras has offered technical advantages, as it allows one to deal with finer topologies relevant to geometry whilst retaining a close link to Hilbert space representation theory and thus a kinship with $C^{*}$-algebras.  As such, operator algebras carrying a completely isometric involution have made recurrent appearances in the study of the unbounded Kasparov product. These are the {\em operator $*$-algebras}. Moreover, the extra flexibility present in the category of operator modules (compared to Hilbert $C^*$-modules) has allowed for recent advances in representation theory, \cite{CrHi:PCO}.

Here we will introduce and study what seems to be the most relevant class of inner product modules
over such algebras: the {\em operator $*$-correspondences}. We will provide many interesting examples of these from noncommutative differential
geometry.    We will also give abstract characterizations of operator $*$-algebras and operator $*$-correspondences that reflect what is seen 
in some of our examples, and use this to study e.g.\ the linking operator $*$-algebra of an operator $*$-correspondence.
Further motivation for studying operator $*$-correspondences will be given below.

An operator $*$-algebra is an operator algebra with an
involution $\dagger$ making it a
$*$-algebra with
$\Vert [a_{ji}^\dagger ] \Vert = \Vert [a_{ij} ] \Vert$.
Their importance in noncommutative differential
geometry was first discovered by the third author in
his Ph.\ D.\ thesis (see \cite{Mes}),
who was soon joined by the second author with Lesch
\cite{KaLe:Flow}.  
Many 
examples of operator $*$-algebras occur naturally
in noncommutative differential
geometry, as we shall see.   In passing we remark  that in fact when one looks for them
elsewhere one sees that examples of operator $*$-algebras  are fairly common in general operator algebra theory,
but this seems hitherto to have not been noticed.

It is natural to investigate for operator $*$-algebras the appropriate analogue of Hilbert $C^*$-modules, and their
bimodule variant, which  these days are often called {\em $C^*$-correspondences}. The theory of $C^{*}$-correspondences has a central place in $C^{*}$-algebra theory, and goes back to Paschke and Rieffel, \cite{Pas:IPM,Rie:IRA}. Remarkably, a countably generated Hilbert $C^{*}$-module $X$ over a $C^{*}$-algebra $B$ is projective in the strong sense of Kasparov's stabilization theorem \cite{Kas1}. The close relationship between the existence of an inner product on $X$ with values in $B$ and direct summands of an appropriately defined free module has been the basis of the first proposals for generalizing Hilbert $C^*$-modules to the setting of modules over a
nonselfadjoint operator algebra $\C B$, due to the first author (see \cite{Ble:Hilb, Ble:Gen}).

In the case where the nonselfadjoint operator algebra $\C B$ has a contractive approximate identity, the existence of an asymptotic factorization of the module through free finite rank modules $\C B^n$ has played a significant role \cite{Ble:Gen, BMP}, see for example Theorem 3.1 in \cite{Ble:Hilb} or Definition 3.1 in \cite{Ble:Gen} (although assumptions (3) and (4) in that definition were removed in later work of the first author and his coauthors). The existence of such a factorization is equivalent to the existence of a contractive approximate identity in the appropriate abstract operator algebra of so-called ``compact'' operators, plus that the latter algebra acts nondegenerately on the module. For such ``rigged modules'' the abstract operator algebra of compact operators is completely isometrically isomorphic to the corresponding represented algebra of module endomorphisms. 

In \cite{Mes}, the rigged modules of \cite{Ble:Gen} were considered in the context of spectral triples, and inner products with values in a relevant operator $*$-algebra were shown to arise naturally. Various analogues and intrinsic definitions in the general operator $*$-algebra setting have been proposed starting in \cite{KaLe:Flow}, and continuing in several papers of the second two authors and their coauthors. 

In important examples arising in noncommutative differential geometry the aforementioned factorizations through free modules need not exist. The paper \cite{Kaa:SBB} shows that, for operator $*$-algebras associated to Riemannian manifolds, the existence of a factorization of the identity operator is equivalent to a notion of \emph{bounded geometry}. Many natural examples arising in geometry do not possess such bounded geometry.

One such example is the Hopf fibration, which was studied in detail in \cite{BMS}. This study leads to the definition of a class of modules proposed in \cite{MesRen}. Here, factorization of the identity through free modules is no longer present, and the natural representation of the compact operators as module endomorphisms is not completely isometric. However, the closure of the algebra of compact operators inside the module endomorphisms does possess a bounded approximate identity. This setting is broad enough to obtain an existence proof of general unbounded Kasparov products. 

The papers \cite{Kaad:Abs, Kaad:Diff, Kaa:MEU} omit the use of bounded approximate identities altogether, and discuss its consequences for the stabilization theorem, Morita equivalence and twisted unbounded Kasparov products.

In the present paper, we proceed in full generality, not requiring the existence of approximate identities in the definition
of our {\em operator $*$-correspondences}. This has some significant ramifications. In particular we no longer
necessarily have an completely isometric inclusion of the algebra of ``compacts''  inside the completely bounded maps on the module,
nor the usual connection mentioned above with free or finitely generated projective modules. This is somewhat related to the fact that on our inner product modules the inner product does not control the norm, contrary to the Hilbert $C^*$-module case. Another technical obstacle present in the general study of operator spaces (and hence also for operator $*$-correspondences) is that the algebra of completely bounded endomorphisms of an operator space is {\em not} in general an operator algebra (see \cite[Theorem 3.4]{Ble:CBC}). In the paper at hand, we overcome this difficulty by employing the Wittstock-Stinespring representation theorem for completely bounded multilinear maps as developed by Christensen, Effros and Sinclair and by Paulsen and Smith, \cite{ChSi:RCM,ChEfSi:CBM,PaSm:MTO}.

One of the main advantages of our general operator $*$-correspondences is that all of the examples mentioned above fit into the framework of the present paper. Operator $*$-correspondences are less restrictive than Hilbert $C^*$-modules, and yet comprise a categorical framework that is amenable to the same kind of analysis that one does with Hilbert $C^*$-modules, such as direct sums and tensor products. Indeed this is a main reason why one wants to use the matrix norms implicit in our operator $*$-correspondences instead of an equivalent Banach space norm. In a contemporaneous paper of the second author \cite{Kaa:MEU} a notion of Morita equivalence of operator $*$-correspondences in the context of the unbounded Kasparov product is discussed. Several more examples are presented there.

\subsection*{Structure of the paper} In Section 2  we characterize operator $*$-algebras, showing that they arise exactly as those closed subalgebras of $B(H)$ that are $*$-invariant up to conjugation by a symmetry. We point out connections with their generated $C^*$-algebras
such as the $C^*$-envelope.  We also give several examples of
these $*$-algebras, among which those arising from spectral triples,
group $C^*$-algebras, and differentiable manifolds.

  In Section 3 we discuss {\em operator $*$-modules}, define their
``compact operators'', and introduce operator $*$-correspondences.

Section 4 contains a detailed discussion of several examples of operator $*$-correspondences from (non)commutative differential geometry. In particular we show how  operator $*$-correspondences arise
as ``minimal domains'' of hermitian connections. 
For example, differentiable sections of hermitian vector bundles over a Riemannian manifold
with a hermitian connection may be given the structure of an operator $*$-correspondence. Other examples come from fiber bundles
equipped with a Riemannian metric in the fiber direction and a given ``vertical connection'',
$C^1$-versions of crossed products of a Riemannian manifold by a discrete group
related to the construction of the Baum-Connes assembly map, hermitian connections on Hilbert $C^{*}$-modules,
and a special case of the latter involving a differentiable version of the Pimsner algebra associated to a self-Morita equivalence bimodule. All of these examples of operator $*$-correspondences admit a closed non-inner product preserving embedding into an appropriate Hilbert $C^{*}$-module. The inner product coincides with that in the ambient Hilbert $C^{*}$-module \emph{up to conjugation by a symmetry}.
 
In Section 5, we formalize the latter observation and relate it to our characterization result for operator $*$-algebras. For operator $*$-correspondences, we introduce a notion of {\em standard form representation} relative to pairs of standard form representations of the
acting operator $*$-algebras.  We prove that operator $*$-correspondences always admit such a standard form representation, explaining the particular form of the examples in Section 4. Our final and main result relates the standard form of operator $*$-algebras and operator $*$-correspondences directly. We show that, up to completely bounded isomorphism, operator $*$-correspondences are the inner product bimodules of operators on a Hilbert space where the inner product agrees with the pairing $(T, S) \mapsto T^* S$
up to conjugation of $T^*$ by a symmetry. This in turn can be used to construct a linking operator $*$-algebra, from which it follows that operator $*$-correspondences can be alternatively characterized as appropriate ``corners'' of operator $*$-algebras. This is a reprise of the fact that Hilbert $C^*$-modules and $C^*$-correspondences may be viewed as ``corners'' of an appropriate $C^*$-algebra.

\subsection*{Acknowledgements}
We are grateful to the Mathematisches Forschungsinstitut Oberwolfach (MFO) for hosting the miniworkshop on \emph{Operator spaces and Noncommutative Geometry in Interaction} in February 2016. The basic ideas for this paper originate from this mini-workshop. 

The first author was partially supported by the National Science Foundation (NSF).

The second author was partially supported by the Villum Foundation (grant 7423).

The third author was partially supported by Simons Foundation grant 346300, the Polish Government MNiSW 2015-2019 matching fund and enjoyed the hospitality of the Instytut Matematyczny PAN (IMPAN), Warsaw, Poland.

The second and third author gratefully acknowledge the Syddansk Universitet, Odense, Denmark and the Leibniz Universit\"{a}t Hannover, Germany for their financial support in facilitating this collaboration.

This work benefited from various conversations with Magnus Goffeng and Ryszard Nest.

\section{Operator $*$-algebras}

\subsection{Definitions}

Let $H$ be a Hilbert space and let $V \su B(H)$ be a closed subspace of the bounded operators on $H$ (equipped with the operator norm). The $(n \ti n)$-matrices with entries in $V$, $M_n(V)$, can then be identified with a closed subspace of the bounded operators on $H^{\op n}$, the $n$-fold Hilbert space direct sum of $H$ with itself. In this way $M_n(V) \su B(H^{\op n})$ becomes a normed vector space for all $n \in \nn$. The characterizing properties of this countable family of normed vector spaces are contained in the following abstract definition of an operator space:

\begin{dfn}
A vector space $\C X$ over $\cc$ is an \emph{operator space} when it is equipped with a norm $\| \cd \|_{\C X} : M_n(\C X) \to [0,\infty)$ for all $n\in \nn$ such that the following holds:
\begin{enumerate}
\item $\C X \cong M_1(\C X)$ is complete in the norm $\| \cd \|_{\C X} : M_1(\C X) \to [0,\infty)$;
\item For each $n \in \nn$, $\xi \in M_n(\C X)$ and $\la,\mu \in M_n(\cc)$ we have the inequality
\[
\| \la \cd \xi \cd \mu \|_{\C X} \leq \| \la \|_{\cc} \cd \| \xi \|_{\C X} \cd \| \mu \|_{\cc},
\]
where $(\la \cd \xi \cd \mu)_{ij} = \sum_{k,l = 1}^n \la_{ik} \cd \xi_{kl} \cd \mu_{lj}$ for all $i,j \in \{1,\ldots,n\}$ and where the norm $\| \cd \|_{\cc}$ on $M_n(\cc)$ is the unique $C^*$-algebra norm.
\item For each $n,k \in \nn$, $\xi \in M_n(\C X)$ and $\eta \in M_k(\C X)$ we have that
\[
\| \xi \op \eta \|_{\C X} = \max\{ \| \xi \|_{\C X} , \| \eta \|_{\C X} \},
\]
where $\xi \op \eta \in M_{n + k}(\C X)$ refers to the direct sum of matrices.
\end{enumerate}
A linear map $\phi : \C X \to \C Y$ between two operator spaces is \emph{completely bounded} when there exists a constant $C \geq 0$ such that
\[
\| \phi(\xi) \|_{\C Y} \leq C \cd \| \xi \|_{\C X},
\]
for all $\xi \in M_n(\C X)$ and all $n\in \nn$. The \emph{completely bounded norm} of a completely bounded map $\phi : \C X \to \C Y$ is defined by
\[
\| \phi \|_{\T{cb}} := \inf\big\{ C \in [0,\infty) \mid \| \phi(\xi) \|_{\C Y} \leq C \cd \| \xi \|_{\C X} \T{ for all } \xi \in M_n(\C X) \, , \, \, n \in \nn \big\} \, .
\]
We will also refer to $\|\cdot\|_{\T{cb}}$ as the \emph{cb-norm}.
\end{dfn}

It is a fundamental theorem due to Ruan that any abstract operator space $\C X$ is completely isometrically isomorphic to a closed subspace of bounded operators on a Hilbert space $H$, \cite[Theorem 3.1]{Rua:SCA}. Thus, there exist a closed subspace $V \su B(H)$ together with an isometry $\phi : \C X \to V$ which induces an isometry $\phi : M_n(\C X) \to M_n(V)$ for all $n \in \nn$.

\begin{dfn}\label{d:opealg}
An operator space $\C A$ is an \emph{operator algebra} when it comes equipped with a multiplication $\cd : \C A \ti \C A \to \C A$ such that
\begin{enumerate}
\item $\C A$ becomes an algebra over $\cc$;
\item The following estimate holds
\[
\| x \cd y \|_{\C A} \leq \| x \|_{\C A} \cd \|y\|_{\C A} \q \T{for all } n \in \nn \T{ and } x,y \in M_n(\C A),
\]
where $(x \cd y)_{ij} = \sum_{k = 1}^n x_{ik} \cd y_{kj}$ for all $i,j \in \{1,\ldots,n\}$.
\end{enumerate}
We say that two operator algebras $\C A$ and $\C B$ are \emph{cb-isomorphic} when there exists a completely bounded and invertible algebra homomorphism $\phi : \C A \to \C B$ such that $\phi^{-1} : \C B \to \C A$ is completely bounded as well. In this case we say that $\phi : \C A \to \C B$ is a \emph{cb-isomorphism}.
\end{dfn}

The above abstract definition corresponds, up to cb-isomorphism, to concrete objects. Any closed subalgebra $\C B$ of $B(H)$ can be given the structure of an operator algebra, where the algebra structure is induced from $B(H)$ and where the matrix norms $\| \cd \|_{\C B} : M_n(\C B) \to [0,\infty)$, $n \in \nn$, come from the operator norm on $B(H^{\op n})$. It then follows by \cite[Theorem 2.2]{Ble:CBC} that any operator algebra $\C A$ is cb-isomorphic to a closed subalgebra $\C B \su B(H)$ for some Hilbert space $H$. 

This result can be strengthened in the case where the operator algebra $\C A$ has a (two-sided) contractive approximate identity. Indeed, in this case there exists a Hilbert space $H$ and a completely \emph{isometric} algebra homomorphism $\C A \to B(H)$, see \cite[Theorem 3.1]{BRS:COA} for the unital case and \cite[Theorem 2.2]{Rua:CNO} for the approximately unital case.

\begin{remark}
We are in this text deviating slightly from the standard terminology, see for example \cite[Section 2.1]{BlMe:OMO}. By an \emph{operator algebra} we will always understand an operator algebra in the sense of Definition \ref{d:opealg}. In particular, we will \emph{not assume} the existence of a completely isometric homomorphism $\C A \to B(H)$. We will also \emph{not assume} the existence of contractive approximate identities (not even bounded approximate identities). 
\end{remark}

The following definition of an ``operator algebra with involution'' can be found in \cite{Mes}. The terminology is taken from \cite{KaLe:Flow}.

\begin{dfn}
An operator algebra $\C A$ is an \emph{operator $*$-algebra} when it comes equipped with an involution $\da : \C A \to \C A$ such that
\begin{enumerate}
\item $\C A$ becomes a $*$-algebra;
\item The involution is a complete isometry, thus
\[
\| x^\da \|_{\C A} = \| x \|_{\C A} \q \T{for all } n \in \nn \T{ and } x \in M_n(\C A),
\]
where $(x^\da)_{ij} = (x_{ji})^\da$ for all $i,j \in \{1,\ldots,n\}$.
\end{enumerate}
We say that two operator $*$-algebras $\C A$ and $\C B$ are \emph{cb-isomorphic} when there exists a cb-isomorphism $\phi : \C A \to \C B$ of the underlying operator algebras such that $\phi(a^\da) = \phi(a)^\da$ for all $a \in \C A$.
%
\end{dfn}

Any $C^*$-algebra $A$ can be given the structure of an operator $*$-algebra with matrix norms $\| \cd \|_A : M_n(A) \to [0,\infty)$, $n \in \nn$, given by the unique $C^*$-norms on $M_n(A)$, $n \in \nn$. It might be tempting to believe that these are the only examples of operator $*$-algebras. To the contrary, the preliminary collection of examples given below shows that the $C^{*}$-algebras form a proper subclass of the operator $*$-algebras. Operator $*$-algebras can be used to capture refined analytic content arising from interesting geometric situations. More examples will be given in Section \ref{s:opecor}.

\subsection{Examples}

\begin{example}\label{Ex: An}[Analytic elements] Operator $*$-algebras arise in the context of algebras of analytic elements inside a $C^{*}$-algebra $A$. Let $\tau:\mathbb{R}\to \mbox{Aut} (A)$ be a 
strongly continuous $1$-parameter group of $*$-automorphisms of $A$. Consider the closed strip
\[
I:= \big\{ z \in \cc \mid \T{Im}(z) \in [-1,1] \big\} \, ,
\]
and denote by $I^{\circ}$ its interior. We define a $*$-subalgebra $\mathcal{A}\subseteq A$ as follows
\[
a \in \C A \lrar \Big( \exists \T{ continuous } f_a : I \to A \T{ s.t. } f_a : I^\circ \to A \T{ analytic and } f_a(t) = \tau_t(a) \, \, \forall t \in \rr \Big) .
\]
We define the unbounded operators $\tau_+ : \C A \to A$ and $\tau_- : \C A \to A$ by $\tau_{\pm}(a) := f_a(\pm i)$. By the uniqueness of analytic extensions, this definition is unambiguous and it follows moreover that $\tau_{\pm}$ are algebra homomorphisms with $\tau_{\pm}(a^{*})=\tau_{\mp}(a)^{*}$ for all $a \in \C A$.  Moreover, by the Phragm\'en-Lindel\"of theorem we have that 
\begin{equation}\label{eq:phraineq}
\| a \| \leq \sup_{z \in I}\| f_a(z) \| \leq \max \{ \| \tau_+(a) \|, \| \tau_-(a) \| \} \q \forall a \in \C A \, .
\end{equation}
In particular, we have an injective algebra homomorphism
\[ 
\pi:\mathcal{A}\to M_{2}(A),\quad a\mapsto \begin{pmatrix} \tau_{+}(a) & 0 \\ 0 & \tau_{-}(a)\end{pmatrix} \, ,
\]
which is moreover closed as an unbounded operator from $A$ to $M_2(A)$. Indeed, to see that $\pi$ is closed we suppose that we have a sequence $\{ a_n\}$ in $\C A$ such that $a_n \to a$ and $\pi(a_n) \to b$ for some $a \in A$ and $b \in M_2(A)$. From the last inequality of \eqref{eq:phraineq} we then see that the sequence $\{ f_{a_n} \}$ of functions from $I \to A$ converges uniformly to a function $f_a : I \to A$. Clearly, $f_a : I \to A$ is continuous and the restriction $f_a : I^\circ \to A$ is analytic. This shows that $a \in \C A$ and that $\pi(a) = b$ and thus that $\pi$ is closed. Hence $\pi(\mathcal{A})$ is a closed subalgebra of $M_{2}(A)$. The identity
\[
\pi(a^{*})=\begin{pmatrix} \tau_{+}(a^{*}) & 0 \\ 0 & \tau_{-}(a^{*})\end{pmatrix}=\begin{pmatrix} \tau_{-}(a)^{*} & 0 \\ 0 & \tau_{+}(a)^{*}\end{pmatrix} \, ,
\]
then shows that the  $*$-algebra $\mathcal{A}$ becomes an operator $*$-algebra when equipped with the matrix norms
\[
\| a \|_{\C A} := \| \pi(a) \|_{M_2(A)}, \q a \in M_n(\C A) \, .
\]

\end{example}

\begin{example}\label{Ex:der}[General closed derivation]
Let $C$ be a $C^*$-algebra and let $\C B_\de \su C$ be a $*$-subalgebra which comes equipped with a closed derivation $\de : \C B_\de \to C$ with $\de(b^*) = -\de(b)^*$ for all $b \in \C B_\de$. We refer to such a derivation as a closed $*$-derivation. Define the injective algebra homomorphism
\[
\rho : \C B_\de \to M_2(C), \q b \mapsto \ma{cc}{b & 0 \\ \de(b) & b}, 
\]
and remark that
\[
\rho(b^*) = \ma{cc}{0 & i \\ -i & 0} \rho(b)^* \ma{cc}{0 & i \\ -i & 0}.
\]
It follows from this relation that $\C B_\de$ becomes an operator $*$-algebra when equipped with the $*$-algebra structure inherited from $C$ and with the matrix norms defined via $\rho : \C B_\de \to M_2(C)$. This operator $*$-algebra structure is used in the papers \cite{Kaad:Abs,Kaad:Diff} for a range of constructions in noncommutative geometry. The identities
\[\begin{pmatrix} 0 & 0 \\ \delta(a) & 0 \end{pmatrix}=\begin{pmatrix} 0 & 0 \\ 0 & 1 \end{pmatrix}\begin{pmatrix} a & 0 \\ \delta(a) & a \end{pmatrix}\begin{pmatrix} 1 & 0 \\ 0 & 0 \end{pmatrix},\quad \begin{pmatrix} a & 0 \\ 0 & 0 \end{pmatrix}=\begin{pmatrix} 1& 0 \\ 0 & 0 \end{pmatrix}\begin{pmatrix} a & 0 \\ \delta(a) & a \end{pmatrix}\begin{pmatrix} 1 & 0 \\ 0 & 0 \end{pmatrix},\]
show that the Banach $*$-algebra norm on $\C B_\de$ is equivalent to the norm \begin{equation}\label{eq: Banachstar}\|b\|_\de:=\|b\|+\|\delta(b)\|.\end{equation} This norm has been well studied, see for example \cite{BlCu:DNS, Sak:ODS, BrRo:UDC, Con:NCG, Hilsumbordism}.
\end{example}

\begin{example}\label{ex: spec}[Spectral triples]
Noncommutative geometry provides an ample supply of closed $*$-derivations. Recall that an \emph{ungraded (or odd) spectral triple} $(\mathscr{A},H, D)$ consists of a Hilbert space $H$ and a $*$-subalgebra $\mathscr{A}\su B(H)$, together with a selfadjoint unbounded operator $D:\mbox{Dom}(D)\to H$ such that
\begin{enumerate}
\item $a\cd \mbox{Dom}(D)\su \mbox{Dom}(D),$ and $[D,a]:\T{Dom}(D)\to H$ extends to a bounded operator on $H$;
\item $a \cd (i + D)^{-1} : H \to H$ is a compact operator,
\end{enumerate}
for all $a\in\mathscr{A}$, \cite{Con:GSV}. Denote by $A$ the $C^{*}$-algebra obtained as the closure of $\mathscr{A}$ in the operator norm on $B(H)$. For the discussion below, only condition (1) is needed. The map 
\[
\delta:\mathscr{A}\to B(H), \quad a\mapsto \overline{[D,a]}\in B(H),
\]
is a closable $*$-derivation, where $\overline{[D,a]}$ denotes the bounded extension to $H$ of $[D,a] : \T{Dom}(D) \to H$. As in Example \ref{Ex:der} we obtain an operator $*$-algebra $\C A_\de$ by passing to the closure of $\de : \sA \to B(H)$. We refer to this operator $*$-algebra as the {\it minimal} operator $*$-algebra associated to $(\sA,H,D)$.

Alternatively, we may define the $*$-subalgebra $\Lip(\delta) \su A$ as follows: 
\[
\begin{split}
& a \in \Lip(\delta) \lrar  \\
& \q \Big( a \cd \mbox{Dom}(D)\su \mbox{Dom}(D) \T{ and } [D,a]:\T{Dom}(D)\to H \T{ extends to a bounded operator on } H \Big).
\end{split}
\]
The map
\[
\de : \Lip(\de) \to B(H), \quad a \mapsto \overline{[D,a]}\in B(H),
\]
is then a closed $*$-derivation and we thus obtain an operator $*$-algebra structure on $\Lip(\de)$ (as in Example \ref{Ex:der}). We refer to this operator $*$-algebra as the {\it maximal} operator $*$-algebra or the {\it Lipschitz algebra} associated to $(\sA,H,D)$. We remark that $\C A_\de \su \Lip(\de)$ but that equality need \emph{not} hold. Moreover $\Lip(\de)$ need not sit densely inside $B(H)$. 

In the paper \cite{Hilsumbordism}, M. Hilsum introduces the Banach $*$-algebra $\Lip(\de)$ associated to a symmetric operator, using the norm $\|\cdot \|_{\delta}$ defined in Equation \eqref{eq: Banachstar}. Its structure as an operator $*$-algebra has been first used in \cite{Mes} and later in \cite{Kaad:Diff, Kaa:MEU, KaLe:Flow, BMS, MesRen}, in the context of the unbounded Kasparov product.

\end{example}

\begin{example}\label{ex: conbun}[Differentiable manifolds]  Let $M$ be a Riemannian manifold (without boundary), $C^{\infty}_{c}(M)$ the $*$-algebra of smooth compactly supported functions on $M$ and $\Ga_c^\infty(T^*M)$ the $C_c^\infty(M)$-module of smooth compactly supported sections of the cotangent bundle $T^{*}M \to M$. Define the Hilbert $C^*$-module $C_0(M) \op \Ga_0(T^* M)$ over $C_0( M)$ as the completion of $C^\infty_c(M)$-module $C_c^\infty( M) \op \Ga_c^\infty(T^*  M)$ with respect to the inner product 
\[
\binn{\ma{c}{f_0 \\ \om_0 }, \ma{c}{f_1 \\ \om_1}  } := \ov{f_0} \cd f_1 + \inn{\om_0,\om_1}_{T^*  M},
\]
where the pairing $\inn{\cd,\cd}_{T^* M}$ of smooth one-forms comes from the Riemannian metric on $M$. We let $\mbox{End}^*_{C_0(M)}( C_0(M) \op \Ga_0(T^*M))$ denote the $C^*$-algebra of bounded adjointable endomorphisms of the Hilbert $C^*$-module $C_0(M) \op \Ga_0(T^* M)$. The norm on this $C^*$-algebra is the operator norm which we denote by $\| \cd \|_\infty$.

Define an injective algebra homomorphism $\rho : C_c^\infty(M) \to \mbox{End}^{*}_{C_{0}(M)}( C_0(M) \op \Ga_0(T^*M))$ by
\[
\rho(f) \ma{c}{g \\ \om} := \ma{c}{f \cd g \\ df \cd g + f \cd \om}, \q g \in C_0(\C M) \,  , \, \, \om \in \Ga_0(T^* M), 
\]
where $df \in \Ga_c^\infty(T^* \C M)$ denotes the exterior derivative of the smooth compactly supported function $f : \C M \to \cc$.

We define the operator $*$-algebra $C^1_0(M)$ as the completion of $C_c^\infty(M)$ with respect to the matrix norms  
\[
\| f \|_1 := \max\big\{ \| \rho(f) \|_\infty, \, \| \rho(f^\da) \|_\infty \big\}, 
\q f \in M_n(C_c^\infty(M)),
\]
where the involution $\da$ on $C_c^\infty(M)$ (and hence also on $C_0^1( M)$) comes from the point-wise complex conjugation of smooth functions. Notice that $C_0^1(M)$ consists exactly of the continuously differentiable functions on $M$ vanishing at infinity and with exterior derivative vanishing at infinity too. For more details on the operator $*$-algebra $C_0^1( M)$ we refer to \cite[Section 2.3]{KaLe:Flow} and \cite[Proposition 3.4]{Kaa:SBB}. 
\end{example}

\begin{example}\label{ex: group}[Group $C^*$-algebras]
Let $G$ be a discrete countable group and let $\ell^2(G)$ denote the Hilbert space of square summable sequences indexed by $G$. For each $\ga \in G$ we have a unitary operator
\[
U_\ga : \ell^2(G) \to \ell^2(G), \q U_\ga( \de_\tau) := \de_{\ga \cd \tau},
\]
and we recall that the reduced group $C^*$-algebra of $G$, $C_r^*(G)$, is defined as the smallest $C^*$-subalgebra of $B( \ell^2(G))$ such that
\[
U_\ga \in C_r^*(G) \q \T{for all } \ga \in G \, . 
\]

In this context, we record the following two examples of operator $*$-algebras:

\begin{enumerate}
\item Let $d \in \nn$ and let $G := \zz^d$. We define the operator algebra $C_r^{\T{hol}}(\zz^d)$ as the smallest closed subalgebra of the reduced group $C^*$-algebra, $C_r^*(\zz^d)$, such that 
\[
U_\ga \in C_r^{\T{hol}}(\zz^d) \q \T{for all } \ga \in \big( \nn \cup \{0\} \big)^d \, .
\]
We define the selfadjoint unitary operator
\[
W : \ell^2(\zz^d) \to \ell^2(\zz^d), \q W( \de_{ (n_1,\ldots,n_d)}) := \de_{ (-n_1,\ldots,-n_d)} \, .
\]
We then have that
\[
W U_\ga^* W = U_\ga \q \T{for all } \ga \in \big( \nn \cup \{0\} \big)^d \, ,
\]
where the $*$ refers to the involution in $C_r^*(\zz^d)$. We thus have a well-defined completely isometric involution
\[
\da : C_r^{\T{hol}}(\zz^d) \to C_r^{\T{hol}}(\zz^d), \q x^\da := W x^* W,
\]
and it follows that $C_r^{\T{hol}}(\zz^d)$ is an operator $*$-algebra.

We remark that $C_r^*(\zz^d)$ is isomorphic as a $C^*$-algebra to the continuous functions on the $d$-torus, $C(\mathbb{T}^d)$, via the isomorphism $U_{(n_1,\ldots,n_d)} \mapsto z_1^{n_1} \clc z_d^{n_d}$, where $z_j : \mathbb{T}^d \to \mathbb{T} \su \cc$ denotes the projection onto the $j^{\T{th}}$ factor of the cartesian product. Under this isomorphism $C_r^{\T{hol}}(\zz^d)$ corresponds to the continuous functions $f : \mathbb{D}^d \to \cc$ on the closed poly-disc that are holomorphic on the open poly-disc $( \mathbb{D}^\ci )^d$. In this picture, the involution is given by $f^{\dag}(z)=\overline{f(\overline{z})}$.
\item Let $d \in \nn$ and let $G := \mathbb{F}_d$ be the free group on $d$ generators $v_1,\ldots,v_d \in \mathbb{F}_d$. We define the operator algebra $C_r^{\T{hol}}(\mathbb{F}_d)$ as the smallest closed subalgebra of the reduced group $C^*$-algebra, $C_r^*(\mathbb{F}_d)$, such that
\[
U_{v_j} \in C_r^{\T{hol}}(\mathbb{F}_d) \q \T{for all } j \in \{1,\ldots,d\} \, .
\]
By the universal property of $\mathbb{F}_{d}$, the map
\[
\si:\{v_1,\cdots, v_d\}\to \mathbb{F}_d,\q \si( v_j) := v_j^{-1},
\]
uniquely extends to a group homomorphism $\si:\mathbb{F}_d\to \mathbb{F}_d$. Clearly, $\sigma$ is an automorphism satisfying $\sigma^{2}=\mbox{id}_{\mathbb{F}_{d}}$.

We emphasize that $\si$ does \emph{not} agree with the inverse operation $\ga \mapsto \ga^{-1}$. 
 
 Let $W : \ell^2(\mathbb{F}_d) \to \ell^2(\mathbb{F}_d)$, $W(\de_\ga) := \de_{\si(\ga)}$, denote the associated selfadjoint unitary operator. We then have that
\[
W U_{v_j}^* W \de_{\tau} = W U_{v_j^{-1}} \de_{\si(\tau)} = \de_{\si\big( v_j^{-1} \cd \si(\tau) \big)}  = \de_{v_j \cd \tau} \, ,
\]
for all $\tau \in \mathbb{F}_d$ and $j \in \{1,\ldots,d\}$. But this implies that $W U_{v_j}^* W = U_{v_j}$ and hence that
\[
\da : C_r^{\T{hol}}(\mathbb{F}_d) \to C_r^{\T{hol}}(\mathbb{F}_d), \q x^\da := W x^* W,
\]
is a well-defined completely isometric involution on $C_r^{\T{hol}}(\mathbb{F}_d)$. We conclude that $C_r^{\T{hol}}(\mathbb{F}_d)$ is an operator $*$-algebra.
\end{enumerate}

\end{example} 

\subsection{The standard form of operator $*$-algebras}

We now show that a large class of operator $*$-algebras can be realized in a completely isometric way as concrete operator algebras in which the involution agrees with the involution in the ambient $C^{*}$-algebra up to conjugation by a symmetry (a selfadjoint unitary). First we make the following general observation.
\medskip
\begin{prop}\label{p:invo}
Let $\A$ be an operator $*$-algebra. Then there exist a Hilbert space $H$, a closed subalgebra $\C B \su B(H)$ and a symmetry $u \in B(H)$ together with a completely bounded isomorphism of operator algebras $\pi : \C A \to \C B$ such that
\[
\pi(a)^* = u \cd \pi(a^\da) \cd u \q \T{for all } a \in \C A,
\]
as an identity in $B(H)$.
\end{prop}
\begin{proof}
By \cite[Theorem 2.2]{Ble:CBC} there exist a Hilbert space $H_0$, a closed subalgebra $\C B' \su B(H_0)$ and a completely bounded isomorphism of operator algebras $\rho : \C A \to \C B'$. We define $\pi : \C A \to B(H_0 \op H_0)$ by $\pi(a) := \ma{cc}{\rho(a) & 0 \\ 0 & \rho(a^\da)^*}$, $\C B := \T{Im}(\pi)$ and $u := \ma{cc}{0 & 1 \\ 1 & 0}$.
\end{proof}

\begin{dfn}
A triple $(H,u,\pi)$ satisfying the conclusion of Proposition \ref{p:invo} is called a \emph{standard form representation} of the operator $*$-algebra $\C A$.
\end{dfn}

We now further investigate the completely isometric theory of standard forms.
\medskip

{\em For the remainder of this section we will consider a fixed operator algebra $\A$ and we will assume the existence of a completely isometric algebra homomorphism $\rho : \A \to B(H_0)$ for some Hilbert space $H_0$.}
\medskip

Under this condition we have a completely isometric version of Proposition \ref{p:invo}.
\begin{prop}\label{p:invo1}
Suppose that $\A$ is an operator $*$-algebra with completely isometric involution $\da : \A \to \A$. Then there exist a Hilbert space $H$, a completely isometric algebra homomorphism $\pi : \A \to B(H)$ and a symmetry $u \in B(H)$ such that
\[
\pi(a)^* = u \cd \pi(a^\da) \cd u \q \T{for all } a \in \A.
\]
\end{prop}
\begin{proof}
Define $\pi : \A \to B(H_0 \op H_0)$ by $\pi(a) = \ma{cc}{\rho(a) & 0 \\ 0 & \rho(a^\da)^*}$ and let $u := \ma{cc}{0 & 1 \\ 1 & 0}$.
\end{proof}

Recall that the \emph{adjoint} $\A^*$ of the operator algebra $\A$ is defined by
\[
\A^* := \{ a^* \mid a \in \A \}.
\] 
The adjoint $\A^*$ becomes an operator algebra when equipped with the algebra structure
\[
a^* \cd b^* := (b \cd a)^* ,\q  \la \cd a^* + \mu \cd b^* = (\ov \la \cd a + \ov \mu \cd b)^*,
\]
and with the matrix norms $\| \cd \|_{\A^*} : M_n(\A^*) \to [0,\infty)$, $\| a^* \|_{\A^*} := \| a \|_{\A}$, where $(a^*)_{ij} = (a_{ji})^*$ for all $i,j \in \{1,\ldots,n\}$.

For any completely isometric algebra homomorphism $\pi : \A \to B$, where $B$ is a $C^*$-algebra we obtain a completely isometric algebra homomorphism $\pi : \A^* \to B$ defined by $\pi(a^*) := \pi(a)^*$.

\begin{prop}\label{p:invo2}
Suppose that $\A$ is an operator $*$-algebra with completely isometric involution $\da : \A \to \A$. Then there exist a $C^*$-algebra $B$, a completely isometric algebra homomorphism $\pi : \A \to B$ and an order two $*$-automorphism $\si : B \to B$ such that 
\[
\pi(\A^*) = \si\big( \pi(\A) \big) \q \T{and} \q \si( \pi(a^\da)) = \pi(a^*).
\]
\end{prop}
\begin{proof}
This follows by Proposition \ref{p:invo1} by putting $B := B(H)$ and $\si(b) := u\cd b \cd u$ for all $b \in B$.
\end{proof}

Conversely, we have the following:

\begin{prop}
Suppose that $B$ is a $C^*$-algebra, $\si : B \to B$ is an order two $*$-automorphism and $\pi : \C A \to B$ is a completely isometric algebra homomorphism such that $\pi(\C A^*) = \si\big( \pi(\A) \big)$. Then $\C A$ becomes an operator $*$-algebra with completely isometric involution $\da : \C A \to \C A$ defined by $\pi(a^\da) := (\si \ci \pi)(a^*)$ for all $a \in \A$. 
\end{prop}

We are now ready to show that the $C^{*}$-algebra $B$ in Proposition \ref{p:invo2} may be chosen in a universal way. We recall some universal constructions.

\begin{lemma}\label{unit} 
If $\A$ is a non-unital operator $*$-algebra with completely isometric involution $\da : \A \to \A$, then the unitalization $\A^1$ is also an operator $*$-algebra with completely isometric involution $(a,\la)^\da := (a^\da, \ov \la)$.  
\end{lemma}
\begin{proof}  
We recall that the map $\A^1 \to B(G)$, $(a,\la) \mapsto \pi(a) + \la \cd 1_G$, is a completely isometric algebra homomorphism whenever $\pi : \A \to B(G)$ is a completely isometric algebra homomorphism, see \cite[Theorem 3.1]{Meyerunit}. In particular, choosing $\pi : \A \to B(H)$ as in Proposition \ref{p:invo1} we see that the involution $\da : \A^1 \to \A^1$, $(a,\la)^\da := (a^\da, \ov \la)$, is completely isometric.
\end{proof}

Recall that a $C^{*}$-envelope of the operator algebra $\A$ is a pair $(C^{*}_e(\A),\io)$ consisting of a $C^*$-algebra $C^*_e(\A)$ and a completely isometric algebra homomorphism $\io : \A \to C^*_e(\A)$ such that the following holds:
\begin{enumerate}
\item $C^*_e(\A)$ is generated as a $C^*$-algebra by $\io(\A)$;
\item For any pair $(B,j)$ consisting of a $C^*$-algebra $B$ and a completely isometric algebra homomorphism $j : \A \to B$ such that $j(\A)$ generates $B$ as a $C^*$-algebra, there exists a $*$-homomorphism $\si : B \to C_e^*(\A)$ satisfying $\si \ci j = \io$.
\end{enumerate}

The general existence result for $C^*$-envelopes is due to Arveson and Hamana, \cite{Arvsub, ArvsubII, Ham:IEO}.

Remark that the $C^*$-envelope of $\A$ is unique in the sense that for any alternative $C^*$-envelope $(C^{*}_e(\A)',\io')$ there exists a $*$-isomorphism $\si : C^{*}_e(\A)' \to C^{*}_e(\A)$ such that $\si \ci \io' = \io$.

\begin{prop} Suppose that $\A$ is an operator $*$-algebra with completely isometric involution $\da : \A \to \A$. Then there exists an order two automorphism $\si : C_{e}^{*}(\A)\to C_{e}^{*}(\A)$ such that 
\begin{equation}\label{eq:ordtwo}
(\si \ci \io)(\A)=\io( \A^{*} ) \q \T{and} \q \io(a^*) = \si\big( \io(a^\da) \big).
\end{equation}

\end{prop}
\begin{proof}
Define the completely isometric algebra homomorphism $j : \A \to C_e^*(\A)$ by $j(a) := \io(a^\da)^*$. Since $j(\A) \su C_e^*(\A)$ generates $C_e^*(\A)$ as a $C^*$-algebra there exists a $*$-homomorphism $\si : C_e^*(\A) \to C_e^*(\A)$ such that $(\si \ci j)(a) = \io(a)$. But then we have that
\[
\si\big( \io(a^\da) \big) = \si\big(  j(a) \big)^* = \io(a)^* = \io(a^*),
\]
proving the identities in \eqref{eq:ordtwo}. Moreover, $\si$ has order two since
\[
\si^2(j(a)) = (\si \ci \io)(a) = \io(a^\da)^* = j(a),
\]
and since $j(\C A) \su C_e^*(\A)$ generates $C_e^*(\A)$ as a $C^*$-algebra. \end{proof}

\section{Operator $*$-correspondences}\label{s:opecor}
We now describe a class of inner product bimodules over operator $*$-algebras, generalizing the notion of a $C^{*}$-correspondence for a pair of $C^{*}$-algebras, \cite{Rie:IRA,Rie:MEA,Pas:IPM}. Our inner product modules arise naturally in noncommutative geometry, and we provide various examples.

\subsection{Definitions} We first recall the notion of an operator $\C A$-$\C B$-bimodule for general operator algebras $\C A$ and $\C B$.

\begin{dfn}\label{d:opebim}
Let $\C A$ and $\C B$ be operator algebras. An operator space $\C X$ is an \emph{operator $\C A$-$\C B$-bimodule} when the following holds:
\begin{enumerate}
\item $\C X$ is an $\C A$-$\C B$-bimodule; 

\item The inequalities
\[
\| a \cd \xi \|_{\C X} \leq \| a \|_{\C A} \cd \| \xi \|_{\C X}, \q \| \xi \cd b\|_{\C X} \leq \|\xi \|_{\C X} \cd \| b \|_{\C B},
\]
hold for all $n \in \nn$, $\xi \in M_n(\C X)$ and all $a \in M_n(\C A)$, $b \in M_n(\C B)$ where $(a \cd \xi)_{ij} = \sum_{k = 1}^n a_{ik} \cd \xi_{kj}$ and $(\xi \cd b)_{ij} = \sum_{k = 1}^n \xi_{ik} \cd b_{kj}$ for all $i,j \in \{1,\ldots,n\}$.
\end{enumerate}
Two operator $\C A$-$\C B$-bimodules $\C X$ and $\C Y$ are \emph{cb-isomorphic} when there exists a completely bounded module homomorphism $\phi : \C X \to \C Y$ such that the inverse $\phi^{-1} : \C Y \to \C X$ is completely bounded as well. When $\C A = \cc$ we say that $\C X$ is a \emph{right operator $\C B$-module} and when $\C B = \cc$ we say that $\C X$ is a \emph{left operator $\C A$-module}.

\end{dfn}

An operator $\C A$-$\C B$-bimodule $\C X$ is always cb-isomorphic to a concrete object in the following way: There exist a Hilbert space $H$, a completely bounded map $\phi : \C X \to B(H)$ and completely bounded algebra homomorphisms $\te : \C A \to B(H)$ and $\pi : \C B \to B(H)$ such that
\begin{enumerate}
\item The images $\phi(\C X), \te(\C A), \pi(\C B) \su B(H)$ are all closed and $\phi : \C X \to \phi(\C X)$, $\te : \C A \to \te(\C A)$ and $\pi : \C B \to \pi(\C B)$ are all completely bounded isomorphisms.
\item $\te(a) \cd \phi(\xi) = \te(a \cd \xi)$ and $\phi(\xi) \cd \pi(b) = \phi(\xi \cd b)$ for all $\xi \in \C X$, $a \in \C A$ and $b \in \C B$.
\end{enumerate}
See \cite[Theorem 2.2]{Ble:Gen}. 

In the case where $\C A$ and $\C B$ have contractive approximate identities and $\C X$ is non-degenerate in the sense that the sub-bimodules $\C A \cd \C X$ and $\C X \cd \C B \su \C X$ are norm-dense one may obtain that $\phi, \te$ and $\pi$, in the above statement, are \emph{completely isometric}, see \cite[Corollary 3.3]{ChEfSi:CBM} and \cite[Theorem 3.3.1]{BlMe:OMO}. 
\medskip

For operator $*$-algebras we focus on modules with some extra structure:

\begin{dfn}\label{d:herope}
Let $\C B$ be an operator $*$-algebra. A right operator $\C B$-module $\C X$ is an \emph{operator $*$-module over $\C B$} when it comes equipped with a pairing
\[
\inn{\cd,\cd}_{\C X} : \C X \ti \C X \to \C B ,
\]
satisfying the conditions:
\begin{enumerate}
\item $\inn{\xi, \eta \cd b}_{\C X} = \inn{\xi, \eta}_{\C X} \cd b$ for all $\xi,\eta \in \C X$ and $b \in \C B$;
\item $\inn{\xi, \la \cd \eta + \mu \cd \ze}_{\C X} = \inn{\xi, \eta}_{\C X} \cd \la + \inn{\xi, \ze}_{\C X} \cd \mu$ for all $\xi,\eta,\ze \in \C X$ and $\la,\mu \in \cc$;
\item $\inn{\xi,\eta}_{\C X} = \inn{\eta,\xi}_{\C X}^\da$ for all $\xi, \eta \in \C X$;
\item We have the inequality
\[
\| \inn{\xi,\eta}_{\C X} \|_{\C B} \leq \| \xi \|_{\C X} \cd \| \eta \|_{\C X} \q \T{for all } n\in \nn \T{ and } \xi, \eta \in M_n(\C X) \, ,
\]
where $( \inn{\xi,\eta}_{\C X} )_{ij} := \sum_{k = 1}^n \inn{\xi_{ki}, \eta_{kj}}_{\C X}$ for all $i,j \in \{1,\ldots,n\}$.
\end{enumerate}
We refer to the pairing $\inn{\cd,\cd}_{\C X}$ as the \emph{hermitian form} or the \emph{inner product}. Condition $(4)$ will sometimes be referred to as the \emph{Cauchy-Schwarz inequality}. 

We say that two operator $*$-modules $\C X$ and $\C Y$ over $\C B$ are \emph{cb-isomorphic} when there exists a cb-isomorphism $\phi : \C X \to \C Y$ of the underlying right operator $\C B$-modules such that $\inn{\phi(\xi),\phi(\eta)}_{\C Y} = \inn{\xi,\eta}_{\C X}$ for all $\xi,\eta \in \C X$.
\end{dfn}

Any Hilbert $C^*$-module over a $C^*$-algebra $B$ can be given the structure of an operator $*$-module. The matrix norms $\| \cd \|_X : M_n(X) \to [0,\infty)$, $n \in \nn$, are defined by
\[
\| \xi \|_X := \| \inn{\xi,\xi}_X \|_B^{1/2} ,\q \xi \in M_n(X) \, ,
\]
where $\inn{\cd,\cd}_X$ denotes the extension to matrices of the inner product on $X$, see Definition \ref{d:herope} $(4)$.

Contrary to the case of Hilbert $C^*$-modules, it is \emph{not} true that the inner product $\inn{\cd,\cd}_{\C X}$ on a general operator $*$-module controls the norm. Thus, even though we obtain a bounded operator $\inn{\xi,\cd}_{\C X} : \C X \to \C B$ whenever $\xi \in \C X$, it need \emph{not} hold that $\| \inn{\xi,\cd}_{\C X} \|_{\infty} \geq \| \xi \|_{\C X}$ (where $\| \inn{\xi,\cd}_{\C X} \|_\infty \in [0,\infty)$ refers to the operator norm).

\begin{dfn}\label{dual}
Let $\C X$ be an operator $*$-module. By the \emph{adjoint module} we will understand the operator space $\C X^\da = \{ \xi^\da \mid \xi \in \C X\}$ which as a vector space is the conjugate of $\C X$ and with matrix norms defined by
\[
\| \cd \|_{\C X^\da} : \xi^\da \mapsto \| \xi \|_{\C X}, \q \xi \in M_n(\C X) \, ,
\]
where $(\xi^\da)_{ij} = (\xi_{ji})^\da$, $i,j \in \{1,\ldots,n\}$.
\end{dfn}

The adjoint module $\C X^\da$ becomes a left operator $\C B$-module with left action defined by
\[
b \cd \xi^\da := (\xi \cd b^\da)^\da ,\q \xi \in \C X \, , \, \, b \in \C B.
\]

\begin{remark}
The terminology ``operator $*$-module'' was applied in a much more restrictive way in \cite{KaLe:Flow}. The modules referred to as ``operator $*$-modules'' in the present text were called ``hermitian operator modules'' in \cite{KaLe:Flow}  (except that the inner product was only assumed to satisfy a completely bounded version of the Cauchy-Schwarz inequality).
\end{remark}

\subsection{Compact operators}
Let $\C B$ be an operator algebra. For a right operator $\C B$-module $\C X$ and a left operator $\C B$-module $\C Y$ we recall that the balanced Haagerup tensor product $\C X \wot_{\C B} \C Y$ is the operator space defined as follows: Equip the algebraic tensor product $\C X \ot \C Y$ with the matrix norms
\[
\begin{split}
& \| \cd \|_{\C X \wot \C Y} : M_n( \C X \ot \C Y ) \to [0,\infty) \\
& \| z \|_{\C X \wot \C Y} := \inf\big\{ \| x \|_{\C X} \cd \| y \|_{\C Y} \mid z = x \ot y \big\} \, ,
\end{split}
\]
where $(x \ot y)_{ij} := \sum_{k = 1}^m x_{ik} \ot y_{kj}$ whenever $x \in M_{n,m}(\C X)$ and $y \in M_{m,n}(\C X)$. The corresponding completion $\C X \wot \C Y$ is an operator space known as the {\em Haagerup tensor product} of the operator spaces $\C X$ and $\C Y$. The {\em balanced Haagerup tensor product}, $\C X \wot_{\C B} \C Y$, is obtained from $\C X \wot \C Y$ as the quotient operator space
\[
\C X \wot_{\C B} \C Y := (\C X \wot \C Y)/N \, ,
\]
where $N \su \C X \wot \C Y$ is defined to be the closure of the subspace 
\[
\T{span}_{\cc}\{ x\cd b \ot y - x \ot b \cd y \mid x \in \C X \, , \, \, y \in \C Y \, , \, \, b \in \C B \} \su \C X \wot \C Y.
\]
For a Hilbert $C^{*}$-module $X$ over a $C^*$-algebra $B$, the $C^{*}$-algebra $\mathbb{K}(X)$ of compact operators on $X$ can be identified with the balanced Haagerup tensor product $X\wot_{B}X^\da$ by \cite[Theorem 3.10]{Ble:Hilb}. Indeed, the relevant completely isometric isomorphism is induced by $x \ot_B y^\da \mapsto \ket{x}\bra{y}$, $x,y \in X$, where $\ket{x}\bra{y} \in \mathbb{K}(X)$ is the compact operator defined by $\ket{x}\bra{y} : z \mapsto x \cd \inn{y,z}_X$ for all $z \in X$. This result motivates the following definition:

\begin{dfn}
Let $\C X$ be an operator $*$-module. We refer to the balanced Haagerup tensor product of $\C X$ and $\C X^\da$ over $\C B$:
\[
\mathbb{K}(\C X) := \C X \wot_{\C B} \C X^\da,
\]
as the \emph{compact operators} on $\C X$.
\end{dfn}

We emphasize that we have defined $\mathbb{K}(\C X)$ as an abstract operator space and not via an action on $\C X$. It can then be verified that $\mathbb{K}(\C X)$ is an operator $*$-algebra:

 \begin{prop}
Let $\C X$ be an operator $*$-module. The compact operators on $\C X$ form an operator $*$-algebra with $*$-algebra structure induced by
\[
\begin{split}
& (\xi_0 \ot_{\C B} \eta_0^\da ) \cd (\xi_1 \ot_{\C B} \eta_1^\da) := \xi_0 \cd \inn{\eta_0,\xi_1}_{\C X} \ot_{\C B} \eta_1^\da ,\\
& (\xi_0 \ot_{\C B} \eta_0^\da)^\da := \eta_0 \ot_{\C B} \xi_0^\da \, ,
\end{split}
\]
for all $\xi_0,\xi_1,\eta_0,\eta_1 \in \C X$.
\end{prop} 

In line with the usual constructions regarding Morita equivalence, we then have that $\C X^\da$ is an operator $*$-module over $\mathbb{K}(\C X)$:

\begin{prop}
The adjoint module $\C X^\da$ is an operator $*$-module over $\bK(\C X) = \C X \wot_{\C B} \C X^\da$. The inner product is given by
\[
\inn{\xi^\da,\eta^\da}_{\C X^\da} := \xi \ot_{\C B} \eta^\da, \q \xi,\eta \in \C X,
\]
and the right-module structure is induced by
\[
\xi^\da \cd (\eta \ot_{\C B} \ze^\da) := (\ze \cd \inn{\eta,\xi}_{\C X} )^\da,  \q \xi,\eta,\ze \in \C X.
\]
\end{prop}

\subsection{Operator $*$-correspondences}
Finding the appropriate notion of a bimodule over a pair of operator $*$-algebras requires some care. It is well known that the operator space $\T{CB}(\C X)$ of completely bounded endomorphisms of an operator space $\C X$ is not an operator algebra unless $\C X$ is a column Hilbert space (see \cite[Proposition 5.1.9]{BlMe:OMO}). This is the main reason for applying the abstract definition of $\mathbb{K}(\C X)$ above. More generally, this fact presents an obstruction to a straightforward definition of the analogue of adjointable operators on operator $*$-modules. When the operator $*$-algebra $\mathbb{K}(\C X)$ has a bounded approximate unit, the use of multiplier algebras provides a solution to this problem, as has been employed in \cite{Ble:Gen, KaLe:Flow, Mes, MesRen}. In many situations, notably in case the Riemannian manifold $M$ in Example \ref{ex: conbun} above is not complete, such an approximate unit does not exist. This more general setting has previously been considered in \cite{Kaad:Abs, Kaad:Diff}. The definition of an operator $*$-correspondence given below also appears in \cite{Kaa:MEU}.

\begin{dfn}\label{d:opstarcor}
Let $\C A$ and $\C B$ be operator $*$-algebras and let $\C X$ be an operator $*$-module over $\C B$. We say that $\C X$ is an \emph{operator $*$-correspondence} from $\C A$ to $\C B$ when $\C X$ comes equipped with a left operator $\C A$-module structure such that
 \begin{enumerate}
\item $a \cd ( \xi \cd b ) = (a \cd \xi) \cd b$ for all $a \in \C A$, $\xi \in \C X$, $b \in \C B$;
\item $\inn{a \cd \xi, \eta}_{\C X} = \inn{\xi, a^\da \cd \eta}_{\C X}$ for all $\xi, \eta \in \C X$,  $a \in \C A$.
\end{enumerate}
If $\C X$ and $\C Y$ are operator $*$-correspondences from $\C A$ to $\C B$ and from $\C C$ to $\C D$ respectively. We say that $\C X$ and $\C Y$ are \emph{cb-isomorphic} when there exist a cb-isomorphism $\phi : \C X \to \C Y$ of the underlying operator spaces together with cb-isomorphisms of operator $*$-algebras $\al : \C A \to \C C$ and $\be : \C B \to \C D$ such that
\begin{enumerate}[resume]
\item  $\phi(a \cd  \xi) = \al(a) \cd \phi(\xi)$ and $\phi(\xi \cd b) = \phi(\xi) \cd \be(b)$ for all $a \in \C A$, $\xi \in \C X$, $b \in \C B$; 
\item $\inn{\phi(\xi), \phi(\eta)}_{\C Y} = \be( \inn{\xi,\eta}_{\C X})$ for all $\xi,\eta \in \C X$.
\end{enumerate} 

\end{dfn}

Our first example of an operator $*$-correspondence concerns the left action of the compact operators $\mathbb{K}(\C X)$ on an operator $*$-module $\C X$. 

\begin{prop}\label{prop:compactcor}
Let $\C X$ be an operator $*$-module. Then $\C X$ becomes an operator $*$-correspondence from $\mathbb{K}(\C X)$ to $\C B$ when equipped with the left action induced by:
\begin{equation}\label{eq:compactrep}
x \ot_{\C B} y^\da : z \mapsto x \cd \inn{y , z}_{\C X}.
\end{equation}
\end{prop}
\begin{proof} We will only verify that the left action satisfies the inequality 
\begin{equation}\label{eq:inequ}
\|K \cd z \|_{\C X} \leq \| K \|_{\bK(\C X)} \cd \| z \|_{\C X},
\end{equation}
for all $n \in \nn$, $K \in M_n(\bK(\C X))$ and $z \in M_n(\C X)$. The remaining properties follow by straightforward algebraic manipulations. 

Let thus $x,y \in M_{n,m}(\C X)$ and $z \in M_n(\C X)$ for some $n,m \in \nn$ be given. We then have that
\[
(x \ot_{\C B} y^\da) \cd z = x \cd \inn{y,z}_{\C X},
\]
and hence that
\[
\| (x \ot_{\C B} y^\da) \cd z \|_{\C X} \leq \| x \|_{\C X} \cd \| \inn{y,z}_{\C X} \|_{\C B} \leq \| x \|_{\C X} \cd \| y^\da \|_{\C X^\da} \cd \| z \|_{\C X} \, ,
\]
where we have applied the Cauchy-Schwartz inequality for $\C X$, the right operator $\C B$-module structure on $\C X$, as well as the definition of the operator space structure on $\C X^\da$, see Definition \ref{d:herope}, \ref{d:opebim} and Definition \ref{dual}. But this implies that $\| (x \ot_{\C B} y^\da) \cd z \|_{\C X} \leq \| (x \ot_{\C B} y^\da) \|_{\bK(\C X)} \cd \| z \|_{\C X}$ by the definition of the norm on the balanced Haagerup tensor product. The general inequality in \eqref{eq:inequ} now follows by a density argument. \end{proof}

For an operator $*$-module $\C X$, \cite[Theorem 5.3]{Ble:Gen} shows that if both $\mathbb{K}(\C X)$ and $\C B$ have a contractive approximate identity and act nondegenerately on $\C X$, then $\C X$ is a rigged module. In that case the representation $\mathbb{K}(\C X) \to CB(\C X)$ in Proposition \ref{prop:compactcor} is completely isometric. In particular if the base $\C B$ is a $C^*$-algebra then it follows that $\C X$ is a genuine Hilbert $C^*$-module with respect to a possibly different inner product, \cite[Theorem 3.1]{Ble:Hilb}.  

In general, the abstract operator algebra $\mathbb{K}(\C X)$ and its closure in the representation \eqref{eq:compactrep} can have very different properties. For example, there are cases where $\mathbb{K}(\C X)$ does not have a bounded approximate identity, but where its closure in the representation \eqref{eq:compactrep} does. The reader is referred to  \cite[Definition 3.9, Proposition 3.10 and Remark 3.11]{MesRen} for a detailed discussion of this situation.

Our second example of an operator $*$-correspondence concerns closed right ideals in operator $*$-algebras. Let $\C L \su \C B$ be a closed right ideal in an operator $*$-algebra $\C B$. Define
\[
\C A := \ov{ \T{span}_{\cc} \big\{ \xi \cd \eta^\da \mid \xi, \eta \in \C L \big\} } \su \C B,
\]
Clearly, $\C A$ becomes an operator $*$-algebra when equipped with the $*$-algebra structure and the matrix norms inherited from $\C B$. Moreover, $\C L \su \C B$ becomes an operator $\C A$-$\C B$-bimodule when equipped with the bimodule structure coming from the algebraic operations in $\C B$ and with the matrix norms coming from $\C B$ as well. We then have the following:

\begin{prop}
The $\C B$-valued inner product defined by $\inn{\xi,\eta}_{\C L} := \xi^\da \cd \eta$ for all $\xi,\eta \in \C L$, provides the operator $\C A$-$\C B$-bimodule $\C L$ with the structure of an operator $*$-correspondence from $\C A$ to $\C B$.
\end{prop}

\section{Examples}
In this section we give a range of examples arising from the geometry of Riemannian manifolds, the action of a discrete group thereon, connections on abstract Hilbert $C^{*}$-modules and Pimsner algebras. 
The reason for presenting them here is their formal similarity, which is related to the standard form of operator $*$-correspondences discussed in Section \ref{s:charcor}.

\subsection{Differentiable sections of vector bundles}\label{ss:difsec}
Let $M$ be a Riemannian manifold (without boundary) and let $E \to M$ be a smooth, hermitian complex vector bundle equipped with a hermitian connection $\Na : \Ga^\infty(E) \to \Ga^\infty(E) \ot_{C^\infty(\C M)} \Ga^\infty(T^* M)$. We denote the hermitian form on the smooth sections of $E$ by $\inn{\cd,\cd}_E : \Ga^\infty(E) \ti \Ga^\infty(E) \to C^\infty( M)$ and the exterior derivative with values in smooth one-forms by $d : C^\infty(M) \to \Ga^\infty(T^*M)$. For the convenience of the reader we recall that the hermitianness of $\Na$ can be described by the equation:
\[
d(\inn{s,r}_E) = \big(s,\Na(r)\big)_E + \big(r,\Na(s)\big)_E^\da, \q s,r \in \Ga^\infty(E),
\]
where the pairing $(\cd,\cd)_E : \Ga^\infty(E) \ti \Ga^\infty(E) \ot_{C^\infty(M)} \Ga^\infty(T^*M) \to \Ga^\infty(T^*M)$ is determined by $(s,r \ot \om)_E := \inn{s,r}_E \cd \om$ and the involution $\da$ refers to complex conjugation of one-forms: $(df)^\da := d(\ov f)$.
\medskip

In this example we shall see how to associate an operator $*$-correspondence to this data. To be more precise, the operator $*$-algebra acting from the left will consist of $C^1$-sections of the endomorphism bundle vanishing at infinity whereas the operator $*$-algebra acting from the right consists of $C^1$-functions on the manifold $M$ vanishing at infinity. The operator $*$-correspondence is then given by $C^1$-sections of the vector bundle $E$ vanishing at infinity. We remark that the terminology ``vanishing at infinity'' means that the derivative also vanishes at infinity and this will make our construction depend on the behaviour ``at infinity'' of the hermitian connection and the Riemannian metric. Let us present the details:
\medskip

We recall from Example \ref{ex: conbun} that the notation $C^1_0(M)$ refers to the operator $*$-algebra of $C^1$-functions on $M$ vanishing at infinity and with exterior derivative vanishing at infinity. 
\medskip

We now introduce the relevant operator $*$-algebra of differentiable sections of the endomorphism bundle $\T{End}(E) \to M$:

Define the Hilbert $C^*$-module $\Ga_0(E) \op \big( \Ga_0(E) \hot_{C_0(M)} \Ga_0(T^*M)\big)$ over $C_0(M)$ as the completion of the $C_c^\infty(M)$-module $\Ga_c^\infty(E) \op \big( \Ga_c^\infty(E) \ot_{C_c^\infty(M)} \Ga_c^\infty(T^*M) \big)$ with respect to the inner product
\[
\binn{\ma{c}{s_0 \\ r_0 \ot \om_0}, \ma{c}{s_1 \\ r_1 \ot \om_1}} 
:= \inn{s_0,s_1}_E + \inn{r_0,r_1}_E \cd \inn{\om_0,\om_1}_{T^*M}.
\]
Here we recall that $\inn{\cd,\cd}_{T^* M}$ refers to the hermitian form coming from the Riemannian metric on $M$. We let $\mbox{End}^{*}_{C_{0}(M)}\Big( \Ga_0(E) \op \big( \Ga_0(E) \hot_{C_0(M)} \Ga_0(T^*M) \big) \Big)$ denote the $C^*$-algebra of bounded adjointable endomorphisms of the Hilbert $C^*$-module $\Ga_0(E) \op \big( \Ga_0(E) \hot_{C_0(M)} \Ga_0(T^*M) \big)$. The $C^*$-norm on this $C^*$-algebra is the operator norm denoted by $\| \cd \|_\infty$.

Define the injective algebra homomorphism 
\[
\si : \Ga_c^\infty(\T{End}(E)) \to \mbox{End}^{*}_{C_{0}(M)}\Big( \Ga_0(E) \op \big( \Ga_0(E) \hot_{C_0(M)} \Ga_0(T^*M) \big) \Big),
\]
by the formula
\[
\si : a \mapsto \ma{cc}{ a & 0 \\ \, [ \Na, a  ] & a \ot 1},
\]
where the commutator is really $[\Na, a] := \Na a - (a \ot 1) \Na$. The operator $*$-algebra $\Ga_0^\Na(\T{End}(E))$ is then defined as the completion of $\Ga_c^\infty(\T{End}(E))$ with respect to the matrix norms 
\[
\| a \|_1 := \max\big\{ \| \si(a) \|_\infty, \, \| \si(a^\da) \|_\infty \}, \q a \in M_n\big( \Ga_c^\infty(\T{End}(E)) \big).
\]
Remark that the involution $\da$ on $\Ga_c^\infty(\T{End}(E))$ (and hence also on $\Ga_0^\Na(\T{End}(E))$) is given by the point-wise adjoint operation with respect to the fiber-wise inner products $\inn{\cd,\cd}_{E_x} : E_x \ti E_x \to \cc$, $x \in M$. 
\medskip

We are now ready to introduce our operator $*$-correspondence of differentiable sections of $E \to M$. Recall that we consider the direct sum $C_0(M) \op \Ga_0(T^* M)$ of $C_0(M)$-modules as a Hilbert $C^*$-module over $C_0(M)$, see Example \ref{ex: conbun}.

To each $s \in \Ga_c^\infty(E)$ we may associate two bounded adjointable operators
\[
\begin{split}
\pi(s) & := \ma{cc}{|s\rangle  & 0 \\ |\Na(s)) & |s)} : 
C_0(M) \op \Ga_0(T^*M) \to \Ga_0(E) \op \big( \Ga_0(E) \hot_{C_0(M)} \Ga_0(T^*M) \big)
\q \T{and} \\
\pi^t(s)  & := \ma{cc}{\bra{s} & 0 \\ (\Na(s)| & (s| }:
\Ga_0(E) \op \big( \Ga_0(E) \hot_{C_0(M)} \Ga_0(T^*M) \big)
\to C_0(M) \op \Ga_0(T^*M),
\end{split}
\]
 defined as follows: For $f \in C_c^\infty(M)$, $\eta \in \Ga_c^\infty(T^{*}M)$ and $r \in \Ga_c^\infty(E)$, $\xi \in \Ga_c^\infty(E) \ot_{C^\infty_c(M)} \Ga_c^\infty(T^*M)$ we have 
\[
\pi(s)\ma{cc}{f \\ \eta}  := \ma{cc} {sf \\ \nabla(s)f+s\otimes \eta }\quad\mbox{and}\quad \pi^t(s)\ma{c}{r\\ \xi} := \ma{c}{ \inn{s,r}_E \\ \big(r, \Na(s)\big)_E^\da + (s,\xi)_E}.
\]

We define the operator $*$-correspondence $\Ga_0^\Na(E)$ from $\Ga_0^\Na(\T{End}(E))$ to $C_0^1(M)$ as the completion of $\Ga_c^\infty(E)$ with respect to the matrix norms 
\[
\| s \|_1 := \max\big\{ \| \pi( s ) \|_\infty , \, \| \pi^t(s) \|_\infty \big\} ,\q s \in M_n( \Ga_c^\infty(E)),
\]
where $\| \cd \|_\infty$ refers to the operator norm on the bounded adjointable operators from $C_0(M) \op \Ga_0(T^*M)$ to $\Ga_0(E) \op \big( \Ga_0(E) \hot_{C_0(M)} \Ga_0(T^*M) \big)$ and from $\Ga_0(E) \op \big( \Ga_0(E) \hot_{C_0(M)} \Ga_0(T^*M) \big)$ to $C_0(M) \op \Ga_0(T^*M)$. The inner product on $\Ga_0^\Na(E)$ is induced by the hermitian form $\inn{\cd,\cd}_E$ on $\Ga_c^\infty(E)$ and the bimodule structure is induced by the usual $(\Ga^\infty_c(\T{End}(E)),C_c^\infty(M))$-bimodule structure on $\Ga_c^\infty(E)$. The various estimates needed to verify that $\Ga_0^\Na(E)$ is indeed an operator $*$-correspondence follow from the identities:
\begin{enumerate}
\item $\pi(s \cd f) = \pi(s) \cd \rho(f)$ and $\pi^t(s \cd f) = \rho(f^\da) \cd \pi^t(s)$;
\item $\pi(a \cd s) = \si(a) \cd \pi(s)$ and $\pi^t(a \cd s) = \pi^t(s) \cd \si(a^\da)$;
\item $\pi^t(s) \cd \pi(t) = \rho( \inn{s,t}_E)$, 
\end{enumerate}
which are valid for all $s,t \in \Ga_c^\infty(E)$, $a \in \Ga_c^\infty( \T{End}(E))$ and all $f \in C_c^\infty(M)$. (Recall here that the algebra homomorphism $\rho : C_c^\infty(M) \to \T{End}^*_{C_0(M)}( C_0(M) \op \Ga_0(T^* M))$ was defined in Example \ref{ex: conbun}.)

\subsection{Vertical connections on fiber bundles}
Let $\pi : M \to N$ be a smooth fiber bundle with model fiber $Z$ and let $E \to M$ be a smooth, hermitian complex vector bundle.  We define the \emph{vertical tangent fields} by
\[
\Ga^\infty(T (M/ N) ) := \big\{ X \in \Ga^\infty(T M) \mid X(f \ci \pi) = 0 \, , \, \, \forall f \in C^\infty(N) \big\},
\]
where $T M \to M$ and $T N \to N$ denote the tangent bundles over $M$ and $N$, respectively. The vertical tangent fields can then be identified with the smooth sections of the kernel of the differential $d \pi : T  M \to \pi^*( T  N)$. Thus, $T( M/ N) := \T{Ker}(d \pi)$.

We define the \emph{vertical one-forms} as the $C^\infty(M)$-linear maps:
\[
\Ga^\infty( T^*( M/N) ) := \T{Hom}_{C^\infty( M)}\big( \Ga^\infty(T(M/N)), C^\infty( M)\big).
\]
The vertical one-forms may be identified with the smooth sections of the dual of $T(M/N)$. We recall that the vertical exterior derivative $d_v : C^\infty(M) \to \Ga^\infty( T^*( M/N) )$ is defined by $d_v(f) := (d f)|_{\Ga^\infty(T(M/N))}$ thus by restricting the exterior derivative to the vertical tangent fields.

We suppose that the fiber bundle $\pi : M \to  N$ comes equipped with a Riemannian metric in the fiber direction. To be precise, we suppose that the smooth vector bundle $T(M/ N) \to M$ comes equipped with a hermitian form 
\[
\inn{\cd,\cd}_{T(M/N)} : \Ga^\infty(T (M/N)) \ti \Ga^\infty(T(M/N)) \to C^\infty(M),
\]
such that $\inn{X,Y}_{T(M/N)} \in C^\infty(M,\rr)$ whenever $X,Y$ are real vertical vector fields. 

We suppose that we have a \emph{vertical connection} on $E \to M$:
\[
\Na : \Ga^\infty(E) \to \Ga^\infty(E) \ot_{C^\infty(M)} \Ga^\infty(T^*(M/N)).
\]
Thus, we have the relation:
\[
\Na(s \cd f) = s \ot d_v(f) + \Na(s) \cd f,
\]
for all $s \in \Ga^\infty(E)$ and $f \in C^\infty( M)$. We do not assume that $\Na$ is hermitian in any sense.
\medskip

We shall now see how to define an operator $*$-correspondence $\C X$ which captures the analytic content of the above setting. This time $\C X$ will consist of the sections of a field of Hilbert spaces over $N$ (with model fiber $L^2$-sections of the restriction of $E$ to the model fiber $Z$), moreover these sections are going to be differentiable in the fiber direction (in a precise way determined by the vertical connection $\Na$). The operator $*$-algebra acting on the left is given by those continuous functions on $M$ that are differentiable in the fiber direction and vanish at infinity. The operator $*$-algebra acting on the right will be the $C^*$-algebra of continuous functions on $N$ vanishing at infinity. Let us give the stringent mathematical definitions:
\medskip

For each $x \in N$ we define the manifold $M_x := \pi^{-1}( \{x\})$ by requiring each local trivialization $\phi_U : \pi^{-1}(U) \to U \ti Z$ with $x \in U \su N$ to induce a diffeomorphism $p_2 \ci \phi_U : M_x \to Z$, where $p_2 : U \ti Z \to Z$ is the projection onto the second factor. The inclusion $i_x :  M_x \to  M$ is then smooth and the derivative induces an isomorphism
\[
d (i_x) : \Ga^\infty( T(M_x) ) \to \Ga^\infty\big( T( M/N)|_{M_x} \big),
\]
of $C^\infty(M_x)$-modules. In this way our Riemannian metric in the fiber direction induces a Riemannian metric
\[
\inn{\cd,\cd}_{T(M_x)} : \Ga^\infty( T( M_x) ) \ti \Ga^\infty( T(M_x) ) \to C^\infty(M_x),
\]
on $M_x$. We may thus define the $C^\infty_c( N)$-linear map $\rho : C^\infty_c( M) \to C^\infty_c(N)$ by ``integration over the fiber'':
\[
\rho(f)(x) := \int_{ M_x} f \, d \mu_x, \q f \in C^\infty_c(M) \,  , \, \, x \in  N,
\]
where the measure $\mu_x$ comes from the Riemannian metric on $M_x$. Notice here that $C^\infty_c(M)$ becomes a module over $C^\infty_c(N)$ using the pull-back along the smooth map $\pi : M \to  N$ together with the algebra structure on $C^\infty(M)$. 

We define the Hilbert $C^*$-module $X$ over $C_0(N)$ by taking the completion of $\Ga^\infty_c(E)$ with respect to the $C_0(N)$-valued inner product defined by
\[
\inn{s_0,s_1}_X := \rho( \inn{s_0,s_1}_E ), \q s_0,s_1 \in \Ga^\infty_c(E),
\]
where $\inn{\cd,\cd}_E$ denotes the hermitian form on $\Ga^\infty(E)$. The right-action of $C_0( N)$ on $X$ is induced by the right action of $C_c^\infty(N)$ on $\Ga^\infty_c(E)$ defined using the pull-back along $\pi : M \to N$ and the module structure of $\Ga^\infty(E)$. Similarly, we define the Hilbert $C^*$-module $Y$ over $C_0(N)$ as the completion of $\Ga_c^\infty(E) \ot_{C^\infty_c( M)} \Ga^\infty_c(T^* (M/N))$ with respect to the $C_0(N)$-valued inner product given by
\[
\inn{s_0 \ot \om_0, s_1 \ot \om_1}_Y :=  \rho\big( \inn{s_0, s_1}_E \cd \inn{\om_0,\om_1}_{T^* (M/N)} \big),
\]
where the hermitian form $\inn{\cd,\cd}_{T^*(M/N)}$ on the vertical one-forms is constructed from $\inn{\cd,\cd}_{T( M/N)}$ using the musical isomorphisms $T(M/N) \cong T^*(M/N)$ of vertical tangent vectors and vertical cotangent vectors. The right action of $C_0(N)$ on $Y$ comes from the right action of $C^\infty(M)$ on $\Ga^\infty(T^*(M/N))$ using the pull-back $\pi^* : C_c^\infty(N) \to C^\infty( M)$.

We define the algebra homomorphism
\[
\si : C_c^\infty( M) \to \mbox{End}^{*}_{C_{0}(N)}( X \op Y) ,\q \si(f) := \ma{cc}{f & 0 \\ \, [\Na, f] & f}.
\]
Remark that the commutator $[\Na,f] : \Ga_c^\infty(E) \to \Ga_c^\infty(E) \ot_{C^\infty_c(M)} \Ga^\infty_c(T^* M/N)$ does define a bounded adjointable operator from $X$ to $Y$. Indeed, it follows from the Leibniz rule for $\Na$ that this commutator is simply given by $s \mapsto s \ot d_v f$. We then obtain an operator $*$-algebra $C^{1,v}_0( M)$ as the completion of $C_c^\infty(M)$ with respect to the matrix norms
\[
\| f \|_{1,v} := \max\big\{ \| \si(f) \|_\infty , \, \| \si(f^\da) \|_\infty \big\} ,\q f \in M_n( C_c^\infty(M)).
\]
where $\| \cd \|_\infty$ refers to the operator norm. Notice that the involution $\da$ on $C^{1,v}_0(M)$ and hence also on $M_n(C_0^{1,v}( M))$ comes from complex conjugation of smooth functions on $M$.

We define the linear map
\[
\te : \Ga_c^\infty(E) \to \mbox{Hom}^{*}_{C_{0}(N)}\big( C_0(N) , X \op Y\big),
\q \te(s) = \ma{c}{s \\ \Na(s) }.
\]
where $\mbox{Hom}^{*}_{C_{0}(N)}\big( C_0(N) , X \op Y\big)$ denotes the bounded adjointable operators from $C_0(N)$ to $X \op Y$. The operator $*$-correspondence $\C X$ from $C^{1,v}_0(M)$ to $C_0(N)$ is then given by the completion of $\Ga_c^\infty(E)$ with respect to the matrix norms
\[
\| s \|_{1,v} := \| \te(s) \|_\infty, \q s \in M_n( \Ga_c^\infty(E)).
\]
The inner product on $\C X$ is induced by the hermitian form $\inn{\cd,\cd}_E$ on $\Ga^\infty(E)$. The left action of $C^{1,v}_0(M)$ comes from the module structure of $\Ga^\infty(E)$ over $C^\infty(M)$ whereas the right action of $C_0(N)$ comes from this same module structure in combination with the pull-back along $\pi : M \to N$. The fact that our operator $*$-correspondence satisfies the required norm-estimates relies on the identities:
\begin{enumerate}
\item $\te(s \cd g) = \te(s) \cd g$;
\item $\te(f \cd s) = \si(f) \cd \te(s)$,
\end{enumerate}
which are valid for all $s \in \Ga_c^\infty(E)$, $g \in C_c^\infty(N)$ and $f \in C_c^\infty(M)$.

\subsection{Crossed products by discrete groups}
Throughout this example we let $M$ be a Riemannian manifold and we let $G$ be a discrete countable group. We will assume that we have a right-action $M \ti G \to M$ by diffeomorphisms of $M$. The diffeomorphism associated to a $\ga \in G$ will be denoted by $\phi_\ga : M \to M$, $\phi_\ga(x) = x \cd \ga$. For each $\ga \in G$ we let $(d \phi_\ga)(x) : T_x M \to T_{\phi_\ga(x)} M$ denote the derivative of $\phi_\ga : M \to M$ evaluated at the point $x \in M$. The following extra conditions will be in effect:

\begin{assu}\label{a:grpact}
It will be assumed that
\begin{enumerate}
\item The action of $G$ on $M$ is proper;
\item The action of $G$ on $M$ is isometric, thus
\[
\| (d \phi_\ga)(x) \|_x = 1,
\]
for all $\ga \in G$ and $x \in M$, where $\| \cd \|_x : B( T_x M, T_{\phi_\ga(x)} M) \to [0,\infty)$ refers to the operator norm.
\end{enumerate}
\end{assu}

For each $\ga \in G$ we define the $*$-isomorphism $\al_\ga : C_c^\infty(M) \to C_c^\infty(M)$ by $\al_\ga : f \mapsto f \ci \phi_\ga$. Notice that $\al_\ga \ci \al_\tau = \al_{\ga \tau}$ for all $\ga,\tau \in G$.
\medskip

We are now going to construct an operator $*$-correspondence $\C X$ which links a $C^1$-version of the reduced crossed product of $C_0(M)$ by the group $G$ and the $C^1$-functions vanishing at infinity on the quotient $M/G$. At the $C^*$-algebraic level this kind of correspondence plays a fundamental role for the construction of the Baum-Connes assembly map, see \cite{BaCoHi:CPK,KaSk:GBN}. For more details on the present operator $*$-correspondence and how it links to a differentiable version of the Morita equivalence between crossed products and quotient spaces (in the case where the action is also free) we refer to \cite{Kaa:MEU}.
\medskip

We consider the $*$-algebra $C_c(G, C_c^\infty(M))$ consisting of finite sums $\sum_{\ga \in G} f_\ga U_\ga$ of elements in $C_c^\infty(M)$ indexed by the discrete group $G$ and equipped with the usual convolution $*$-algebra structure.

 We define the Hilbert $C^*$-module $\ell^2\big(G, C_0(M) \op \Ga_0(T^*M)\big)$ over $C_0(M)$ as the exterior tensor product of $C_0(M) \op \Ga_0(T^* M)$ by the Hilbert space of $2$-summable sequences indexed by $G$. Recall here that the Riemannian metric provides $\Ga_0(T^*M)$ with the structure of a Hilbert $C^*$-module over $C_0(M)$, see Example \ref{ex: conbun}. We define the injective algebra homomorphism
\[
\begin{split}
& \si : C_c\big( G, C_c^\infty(M) \big) \to \mbox{End}^{*}_{C_0(M)}\Big(  \ell^2\big(G, C_0(M) \op \Ga_0(T^*M)\big) \Big), \\
& \si( f U_\ga) \big( \ma{c}{h \\ \om} \ot \de_\tau \big)
:= \ma{cc}{\al_{\ga \tau}^{-1}(f) & 0 \\ d\big( \al_{\ga \tau}^{-1}(f) \big) & \al_{\ga \tau}^{-1}(f) } \cd \ma{c}{h \\ \om} \ot \de_{\ga \tau}.
\end{split}
\]
Remark that the fact that $\si(f U_\ga)$ is indeed a bounded adjointable operator relies on condition $(2)$ of Assumption \ref{a:grpact}.

The reduced crossed product operator $*$-algebra $C^1_0(M) \rtimes_r G$ is then defined as the completion of $C_c(G,C_c^\infty(M))$ with respect to the matrix norms:
\[
\| x \|_1 := \max\big\{ \| \si(x) \|_\infty \,  , \, \, \| \si(x^\da) \|_\infty \big\} ,
\q x \in M_n\big( C_c\big( G, C_c^\infty(M) \big)\big),
\]
where the involution $\da$ on $C_c\big( G, C_c^\infty(M) \big)$ and hence also on $C^1_0(M) \rtimes_r G$ is given by the usual involution on $C_c\big( G, C_c^\infty(M) \big)$. To wit: $(f U_\ga)^\da := \al_{\ga}^{-1}(\, \ov{f} \, ) U_{\ga^{-1}}$.
\medskip

We now consider the $*$-algebra $C^\infty_c(M/G)$ consisting of all $G$-invariant smooth maps $f : M \to \cc$ such that the induced map $[f] : M / G \to \cc$ has compact support. Define the algebra homomorphism
\[
\rho : C^\infty_c(M/ G) \to \mbox{End}^{*}_{C_0(M)}\big( C_0(M) \op \Ga_0(T^*M) \big), \q \rho(f) := \ma{cc}{f & 0 \\ df & f}.
\]
Remark that Assumption \ref{a:grpact} implies that the section $df \in \Ga^\infty(T^* M)$ is automatically bounded with respect to the hermitian form $\inn{\cd,\cd}_{T^*M}$ and hence that $\rho(f)$ is a bounded adjointable operator. We define the operator $*$-algebra $C^1_0(M/G)$ as the completion of $C^\infty_c(M/G)$ with respect to the matrix norms
\[
\| f \|_1 := \max\big\{ \| \rho(f) \|_\infty \, , \, \, \| \rho(f^\da) \|_\infty \big\}, \q f \in M_n\big( C^\infty_c(M/G) \big).
\]
where the involution $\da$ on $C^1_0(M/G)$ is induced by the involution on $C^\infty(M)$ given by point-wise complex conjugation. 

We equip the vector space $C_c^\infty(M)$ with the structure of a $C_c(G, C_c^\infty( M))$-$C_c^\infty( M/G)$-bimodule with left action defined by
\[
( f U_\ga ) \cd h :=  f \cd \al_\ga(h) \q  f,h \in C_c^\infty(M) \, , \,\, \ga \in G ,
\]
and with right action induced by the algebra structure on $C^\infty(M)$. Furthermore, we define the pairing
\begin{equation}\label{eq:rinn}
\inn{\cd,\cd} : C_c^\infty(M) \ti C_c^\infty(M) \to C_c^\infty(M/G), \q \inn{g,h} := \sum_{\ga \in G} \al_\ga(\ov{g} \cd h). 
\end{equation}
Remark that this pairing is well-defined since the action of $G$ on $M$ is assumed to be proper.

To construct matrix norms on $C_c^\infty(M)$ we define the linear maps
\[
\begin{split}
& \pi : C_c^\infty(M) \to \mbox{Hom}^{*}_{C_0(M)}\big( C_0(M) \op \Ga_0(T^*M), \ell^2\big( G, C_0(M) \op \Ga_0(T^*M) \big) \big), \\
& \pi(g) \ma{c}{h \\ \om} := \sum_{\ga \in G} \ma{cc}{\al_{\ga}^{-1}(g) & 0 \\ d( \al_{\ga}^{-1}(g)) & \al_\ga^{-1}(g) } 
\cd \ma{c}{h \\ \om} \ot \de_\ga, \\
& \pi^t : C_c^\infty(M) \to \mbox{Hom}^{*}_{C_0(M)}\Big( \ell^2\big( G, C_0(M) \op \Ga_0(T^*M)\big), C_0(M) \op \Ga_0(T^*M) \Big), \\
& \pi^t(g)\Big( \ma{c}{h \\ \om} \ot \de_\ga \Big) := \ma{cc}{\al_\ga^{-1}(g) & 0 \\ d( \al_{\ga}^{-1}(g)) & \al_\ga^{-1}(g) }\cd \ma{c}{h \\ \om} \, ,
\end{split}
\]
where we recall that the notation $\T{Hom}^*_B(X,Y)$ refers to the bounded adjointable operators between two Hilbert $C^*$-modules $X$ and $Y$ over a $C^*$-algebra $B$. Remark that the fact that $\pi(g)$ and $\pi^t(g)$ are indeed bounded (adjointable) operators relies on Assumption \ref{a:grpact}.

We then define the operator $*$-correspondence $\C X$ from $C^1_0(M) \rtimes_r G$ to $C^1_0(M/G)$ as the completion of $C_c^\infty(M)$ with respect to the matrix norms:
\[
\| g \|_1 := \T{max}\big\{ \| \pi(g) \|_\infty \, , \, \, \| \pi^t(g^\da) \|_\infty  \big\}, \q g \in M_n(C_c^\infty(M)),
\]
where the $\da$ refers to the involution on $M_n(C_c^\infty(M))$ coming from complex conjugation of functions (together with the transpose operation on matrices). The inner product on $\C X$ is induced by the pairing $\inn{\cd,\cd}$ defined in \eqref{eq:rinn} and the bimodule structure is induced by the $C_c(G, C_c^\infty(M))$-$C_c^\infty(M/G)$-bimodule structure on $C_c^\infty(M)$. Indeed, the estimates required to show that $\C X$ is an operator $*$-correspondence follow from the algebraic identities:
\begin{enumerate}
\item $\pi^t(g) \cd \pi(h) = \rho( \inn{g,h})$;
\item $\pi(g \cd f) = \pi(g) \cd \rho(f)$ and $\pi^t(g \cd f) = \rho(f^\da) \cd \pi^t(g)$;
\item $\pi(x \cd g) = \si(x) \cd \pi(g)$ and $\pi^t(x \cd g) = \pi^t(g) \cd \si( x^\da)$,
\end{enumerate}
which hold for all $g,h \in C_c^\infty(M)$, $f \in C_c^\infty(M/G)$ and all $x \in C_c\big(G, C_c^\infty(M)\big)$.

\subsection{Hermitian connections on Hilbert $C^{*}$-modules}\label{Ex: Hilher}
Let $C$ be a unital $C^*$-algebra and let $\de : \C B_\de \to C$ be a closed $*$-derivation as in Example \ref{Ex:der}. Let $B \su C$ denote the $C^*$-algebra obtained as the norm-closure of $\C B_\de \su C$ and let $X$ be a Hilbert $C^{*}$-module over $B$. Let $\Om(\C B_\de) \su C$ denote the smallest $C^*$-subalgebra of $C$ such that
\[
1\in \Om(\C B_\de), \q  B \su \Om(\C B_\de) \q \T{and} \q \de(b) \in \Om(\C B_\de) \, \, , \q \forall b \in \C B_\de.
\]
We consider $\Om(\C B_\de)$ as a $C^*$-correspondence from $B$ to $\Om(\C B_\de)$.  Indeed, since $\Om(\C B_\de)$ is a $C^*$-algebra it may be viewed as a Hilbert $C^*$-module over itself and it may then be equipped with the left action of $B$ coming from the inclusion $B \su \Om(\C B_\de)$. 

We recall, from Example \ref{Ex:der}, that $\C B_\de$ is an operator $*$-algebra when equipped with the $*$-algebraic structure inherited from $C$ and with the matrix norms defined via the injective algebra homomorphism
\[
\rho : \C B_\de \to M_2( \Om(\C B_\de) ), \q \rho : b \mapsto \ma{cc}{b & 0 \\ \de(b) & b} \, .
\]

Suppose we are given a dense $\C B_\de$-submodule $\C X\su X$ such that $\inn{x,y}_X \in \C B_\de$ for all $x,y \in \C X$. Suppose moreover that we are given a linear $\de$-connection
\[
\nabla_\de:\C X\to X\hot_{B}\Omega(\C B_\de) \, \, , \q \nabla_\de(x \cd b)=\nabla_\de(x) \cd b+x \otimes\delta(b) \, \, ,\q \forall x \in \C X \, \, , \, \, \, b\in \C B_\de,
\]
where $X \hot_B \Om(\C B_\de)$ denotes the interior tensor product of Hilbert $C^*$-modules. We assume that $\nabla_\de$ is \emph{hermitian}, that is for all $x,y \in \C X$ we have
\[
\langle x \otimes 1, \nabla_\de(y)\rangle - \langle\nabla_\de(x),y\otimes 1\rangle =\delta(\langle x,y\rangle) \, . 
\]

\begin{lemma}\label{l:closure} Let $\delta:\C B_\de \to C$ be a closed $*$-derivation and $\nabla_\de:\C X\to X\hot_{B}\Omega(\C B_\de)$ a hermitian $\de$-connection. Then $\nabla_{\de}$ is closable as an unbounded operator from $X$ to $X\hot_{B}\Omega(\C B_\de)$.
\end{lemma} 
\begin{proof}
Let $\{x_{n}\}$ be a sequence in $\C X$ such that $x_{n}\to 0$ in $X$ and $\nabla_\de(x_{n})\to z$ in $X\hot_{B}\Omega(\C B_\de)$ for some $z \in X\hot_{B}\Omega(\C B_\de)$. We must show that $z=0$. 

For $y\in \C X$ we have
\[
\langle z, y\otimes 1\rangle=\lim_{n \to \infty} \langle \nabla_\de(x_n), y\otimes 1\rangle =\lim_{n \to \infty} \langle x_n\otimes 1,\nabla_\de(y)\rangle  -\delta(\langle x_{n},y\rangle)=-\lim_{n \to \infty} \delta(\langle x_{n},y\rangle) \, ,
\]
from which it follows that $\delta(\langle x_{n},y\rangle)$ is convergent in $C$. Since $\langle x_{n},y\rangle \to 0$ in $B$ and $\delta : \C B_\de \to C$ is closed, it follows that $\delta(\langle x_{n},y\rangle)\to 0$ in $C$ and hence that $\langle z, y\otimes 1\rangle=0$. Remark now that the submodule
\[
\mbox{span}_{\cc}\{y\otimes \omega : y\in\C X, \omega\in \Omega(\C B_{\de})\}\su X\hot_{B}\Omega(\C B_\de) \, ,
\]
is dense in $X\hot_{B}\Omega(\C B_\de)$. Since $\langle z, y\otimes \omega\rangle=\langle z, y\otimes 1\rangle \omega=0$, it holds that $\langle z,w\rangle=0$ for all $w\in X\hot_{B}\Omega(\C B_\de)$ and thus $z=0$.
\end{proof}

For each $\xi \in X\hot_{B}\Om(\C B_\de)$, we define the bounded adjointable operators
\[
\begin{split}
& \ket{\xi} : \Om(\C B_\de)\to X\hot_{B}\Om(\C B_\de) \, \, ,\quad \ket{\xi} : \om \mapsto \xi \cd \om, \\
& \bra{\xi} :   X\hot_{B}\Om(\C B_\de)\to \Om(\C B_\de) \, \, ,\quad \bra{\xi} : \eta \mapsto \inn{\xi,\eta}_{X \hot_B \Om(\C B_\de)}.
\end{split}
\]
Notice that $\ket{\xi}^* = \bra{\xi}$. For each $x \in \C X$, we then define the bounded adjointable operators
\[
\pi(x):=\ma{cc}{
|x \otimes 1\rangle & 0 \\
|\Na_\de(x)\rangle  & |x \otimes 1 \rangle
} : \Om(\C B_\de) \op \Om(\C B_\de)\to (X \hot_B \Om(\C B_\de) ) \op (X \hot_B \Om(\C B_\de) ) .
\]

The completion of the right $\C B_\de$-module $\C X$ with respect to the matrix-norms 
\[
x \mapsto \| \pi(x) \|_\infty \, \, , \q x \in M_n( \C X).
\]
is then an operator $*$-module over the operator $*$-algebra $\C B_\de$. The inner product is induced by the inner product on $X$. We denote this operator $*$-module completion by $\C X_{\Na}$. Remark that the required norm-estimates follow from the algebraic identities, which hold for all $x,y \in \C X$ and all $b \in \C B_\de$:

$$\pi(x \cd b) = \pi(x) \cd \rho(b),\quad U \pi(x)^* U \pi(y) = \rho( \inn{x,y} ),\quad\textnormal{with } U := \ma{cc}{0 & i \\ -i & 0}. $$

Since $\nabla_{\delta} : \C X \to X \hot_B \Om(\C B_\de)$ is closable (by Lemma \ref{l:closure}), the inclusion $\C X \to X$ extends to a completely bounded and injective right-module map $\C X_{\Na} \to X$. The image of $\C X_{\Na}$ in $X$ then agrees with the domain of the closure of $\Na_\de : \C X \to X \hot_B \Om(\C B_\de)$, and we denote the closure by  $$\Na_\de : \C X_\Na \to X \hot_B \Om(\C B_\de),$$ which is now a completely bounded map. The injective $*$-homomorphism
\[
\mbox{End}^{*}_{B}(X)\to \mbox{End}^{*}_{\Omega(\C B_\de)}(X\hot_{B}\Omega(\C B_\de)) \, \, , \q T\mapsto T\ot 1,
\]
allows us to view $\mbox{End}^{*}_{B}(X)$ as a $C^{*}$-subalgebra of $\mbox{End}^{*}_{\Omega(\C B_\de)}(X\hot_{B}\Omega(\C B_\de))$. The operators $$\nabla_{\de} T - (T\ot 1)\nabla_{\de} : \C X_\Na \to X \hot_B \Om(\C B_\de),$$ are $\C B_{\de}$ linear and thus extend to operators
$$\de_{\Na}^{0}(T): \C X_\Na\ot_{\C B_{\de}}^{\textnormal{alg}}\Om(\C B_\de)\to X \hot_B \Om(\C B_\de), $$
via
 \[
\de_\Na^0(T) (x\otimes \omega):= \big((\nabla_\de \cdot T-(T\ot 1)\Na_\de)x \big)\cdot \omega \, .
\]

Consider the subalgebra $\sD$ of $\mbox{End}^{*}_{B}(X)$ defined by  
\[
\begin{split}
T \in \sD & \lrar
\Big( T\in \mbox{End}^{*}_{B}(X) \, , \, \, T\C X_{\nabla}\su \C X_{\nabla} \\
& \qq \T{and} \, \, \, \delta_{\nabla}^{0}(T) \mbox{ admits a bounded extension } \delta_{\nabla}(T): X\hot_{B}\Omega(\C B_\de) \to X\hot_{B}\Omega(\C B_\de) \Big).
\end{split}
\]

Define the $*$-subalgebra $\T{Dom}(\de_\Na) \su \mbox{End}^{*}_{B}(X)$ by
\[
T \in \T{Dom}(\de_\Na) \lrar T, T^* \in \sD.
\]
We wish to show that the map $T\mapsto \delta_{\nabla}(T)$ is a closed $*$-derivation
\begin{equation}\nonumber 
\delta_{\nabla} : 
\mbox{Dom}(\delta_{\nabla}) \rightarrow \mbox{End}_{\Omega(\C B_\delta)}^{*}(X\hot_{B}\Omega(\C B_{\delta})).
\end{equation}
 
Indeed, since $\nabla_\de$ is hermitian we may compute as follows:
 \begin{align*}
\langle \de_\Na(T) ( x\otimes \omega) , y\otimes \eta \rangle  
& = \om^* \cd \langle \Na_\de( Tx) , y \ot 1 \rangle \cd \eta 
- \om^* \cd \langle \Na_\de (x), T^* y \ot 1 \rangle \cd \eta \\ 
& = \om^* \cd \inn{ T x \ot 1, \Na_\de(y) } \cd \eta
- \om^* \cd \de\big( \inn{T x, y}\big) \cd \eta  \\
& \qqq - \om^* \cd \inn{x \ot 1, \Na_\de(T^* y)} \cd \eta 
+ \om^* \cd \de\big( \inn{ x, T^* y} \big) \cd \eta \\
& = - \inn{x \ot \om, \de_\Na(T^*)(y \ot \eta)},
 \end{align*}
for all $x,y \in \C X_\Na$, $\om, \eta \in \Om(\C B_\de)$. Thus, the everywhere defined bounded operator $\delta_{\nabla}(T)$ has an everywhere defined bounded adjoint $-\delta_{\nabla}(T^{*})$, and 
we have a closed $*$-derivation $\de_\Na : \T{Dom}(\de_\Na) \to \mbox{End}_{\Omega(\C B_\delta)}^{*}(X\hot_{B}\Omega(\C B_{\delta}))$. 
As in Example \ref{Ex:der} we equip $\T{Dom}(\de_\Na) \su \T{End}^*_B(X)$ with the structure of an operator $*$-algebra which we denote by $\T{Lip}(\de_\Na)$.

For $T \in \T{Lip}(\delta_{\nabla})$ the operation $x \mapsto Tx$, $x \in \C X_\Na$, then provides the operator $*$-module $\C X_\Na$ with the structure of an operator $*$-correspondence from $\T{Lip}(\de_\Na)$ to $\C B_\de$. Indeed, the fact that the left action satisfies the required estimates follows from the algebraic identity
\[
\ma{cc}{T \ot 1 & 0 \\ \de_\Na(T) & T \ot 1} \cd \ma{cc}{ \ket{x \ot 1} & 0 \\ \ket{\Na_\de(x)} & \ket{x \ot 1}}
= \ma{cc}{ \ket{Tx \ot 1} & 0 \\ \ket{\Na_\de(T x)} & \ket{T x \ot 1}},
\]
of bounded adjointable operators from $\Om(\C B_\de) \op \Om(\C B_\de)$ to $\big( X \hot_B \Om(\C B_\de) \big) \op \big( X \hot_B \Om(\C B_\de) \big)$. For any closed $*$-subalgebra $\A \su \T{Lip}(\delta_{\nabla})$, we thus obtain that $\C X_{\nabla}$ is an operator $*$-correspondence from $\A$ to $\C B_{\de}$.

\subsection{Pimsner algebras}
Let $C$ be a unital $C^*$-algebra and let $\C B_\de \su C$ be a unital $*$-subalgebra which comes equipped with a closed derivation $\de : \C B_\de \to C$ with $\de(b^*) = -\de(b)^*$. We let $B$ denote the unital $C^*$-subalgebra obtained as the completion of $\C B_\de$ in the $C^*$-norm coming from $C$. The unit in $B$ and the unit in $C$ are assumed to agree.

We briefly review the construction of the Cuntz-Pimsner algebra of a self-Morita equivalence bimodule, see \cite{Pim}. 
Thus, we consider a full right Hilbert $C^*$-module $E$ over $B$ together with a $*$-isomorphism
\[
\pi : B \to \mathbb{K}(E)\cong E\hot_{B}E^{*},
\]
turning $E$ into a $C^*$-correspondence from $B$ to $B$. In particular, we obtain that the conjugate module $E^*$ is a $C^*$-correspondence from $B$ to $B$ with right action and inner product given by
\[
\xi^* \cd b := \big(\pi(b^*)( \xi) \big)^* \, \, , \q
\inn{\xi^*,\eta^*} := \pi^{-1}( \xi\otimes \eta^{*}),
\]
for all $\xi,\eta \in E$, $b \in B$. For $n>0$ we use the notation
\[
E_{n}=E^{\hot n}:=\underbrace{E\hot_{B} \cdots \hot_{B} E}_\text{n times},
\]
for the $n$-fold interior tensor product of $E$ with itself. We use the conventions
\[
E_0:=B\quad\mbox{and}\quad E_n:=(E^{*})^{\hot|n|} := \underbrace{E^* \hot_B \cdots \hot_B E^*}_\text{-n times} \quad\mbox{for } n<0,
\]
and write $\pi_{n}:B\to \mbox{End}^{*}_{B}(E_n)$ for the left $B$ representation on $E_{n}$, $n \in \zz$.
The \emph{symmetrized Fock-space} 
\[
 X := \ov{ \bigoplus_{n \in \zz} E_{n}, }
\]
is the Hilbert $C^*$-module completion of the algebraic direct sum $\bigoplus_{n \in \zz} E_{n}$ of the Hilbert $C^*$-modules $E_{n}$, $n \in \zz$. The \emph{Pimsner algebra} $\C O_{E}$ is the $C^{*}$-subalgebra of $\mbox{End}^{*}_{B}(X)$ generated by the \emph{creation operators} $S_{\xi}:X\to X$, for $\xi\in E$, given on the homogenous components by  
 \begin{align}\label{creabea}
 S_{\xi} &: \xi_1 \otimes \ldots \otimes \xi_n \mapsto \xi \otimes \xi_1 \otimes \ldots \otimes \xi_n,\\ 
\nonumber S_{\xi} &:b\mapsto \xi\cdot b ,\\
\nonumber S_{\xi} &: \eta_1^* \otimes \ldots \otimes \eta_n^* \mapsto \inn{\xi^{*},\eta_1^{*}} \cd \eta_2^* \otimes \ldots \otimes \eta_n^* \, .
\end{align}
To construct a differentiable version of the Pimsner algebra $\C O_{E}$ and an associated operator $*$-correspondence, we will need the following:

\begin{assu}\label{a:pimsner}
It will be assumed that
\begin{enumerate}
\item There is a dense $\C B_{\delta}$-bimodule $\mathcal{E}\su E$ such that 
$\inn{\xi,\eta}, \inn{\xi^{*},\eta^{*}}\in \C B_{\delta}$ for all $\xi,\eta\in\mathcal{E}$.
\item There exist a finite number of elements $x_1,\ldots,x_k \in \mathcal{E}$ and $y_1,\ldots,y_l \in \mathcal{E}$ such that
\[
\sum_{j = 1}^k x_j \ot x_j^* = 1_E \q\mbox{and}\q \sum_{i = 1}^l \inn{y_i,y_i} = 1_B .
\]
\end{enumerate}
\end{assu}

We define the adjointable isometries
\begin{equation}\label{eq:isometry}
v_n : E_n \to\fork{ccc}{
B^{\op k^{n}} \q\eta  \mapsto (\langle x_{J},\eta\rangle)_{J\in\{1,\cdots , k\}^{n}} & x_J  := x_{j_1} \otimes \ldots \otimes x_{j_n} & \T{for } n > 0 \\
B^{\op l^{|n|}}  \q\eta^*  \mapsto (\langle y^{*}_{I},\eta^{*}\rangle)_{I\in\{1,\cdots, l\}^{|n|}} & y_I^* := y_{i_1}^* \otimes \ldots \otimes y_{i_{|n|}}^* & \T{for }  n<0,
}
\end{equation}
and let $v_0$ denote the identity operator on $E_0 := B$.

For $n > 0$ we introduce the notation
\[
\C E_n := \underbrace{\C E \ot_{\C B_\de} \cdots \ot_{\C B_\de} \C E}_\text{n times},
\]
and similarly
\[
\C E_0 := \C B_\de \q \T{and} \q 
\C E_n := \underbrace{\C E^* \ot_{\C B_\de} \cdots \ot_{\C B_\de} \C E^*}_\text{-n times} \quad\mbox{for } n<0.
\]
We remark that $\C E_n \su E_n$ for all $n \in \zz$ and that the adjointable isometries above restrict to injective right $\C B_\de$-linear maps
\begin{equation}\label{eq:restrict}
v_n : \C E_n \to \fork{ccc}{ \C B_\de^{\op k^n} & \T{for} & n \geq 0 \\ \C B_\de^{\op l^{|n|}} & \T{for} & n < 0,}
\end{equation}
(where $v_0$ is now the identity operator on $\C B_\de$).

As in Subsection \ref{Ex: Hilher}, let $\Om(\C B_\de) \su C$ denote the smallest $C^*$-subalgebra of $C$ such that 
\[
B \su \Om(\C B_\de) \q \T{and} \q \de(b) \in \Om(\C B_\de) \, \, , \q b \in \C B_\de.
\]

For each $n \in \zz$, we define the \emph{Gra\ss mann connection} $\Na_n : \C E_n \to E_{n} \hot_B \Om(\C B_\de) $  by
\[
\begin{split}\Na_{n}(\xi):=v_{n}^{*} \delta(v_{n}(\xi))&\q \Na_n : \xi \mapsto \fork{ccc}{
\sum_{J \in \{1,\ldots,k\}^n} x_J \otimes \de( \inn{x_J, \xi}) & \T{for} & n >0 \\
1_B \otimes \de(\xi) & \T{for} & n = 0 \\
\sum_{I \in \{1,\ldots,l\}^{|n|}} y_I^* \otimes \de( \inn{y_I^*, \xi}) & \T{for} & n<0.
}
\end{split}
\]
The Gra\ss mann connection $\Na_n : \C E_n \to E_{n} \hot_B \Om(\C B_\de)$ is a closed hermitian $\de$-connection for all $n \in \zz$ (see Subsection \ref{Ex: Hilher} ). We record the following relations between the different Gra\ss mann connections:
\begin{equation}\label{eq:relgrass}
\Na_n = \fork{ccc}{ \sum_{j = 1}^k (S_{x_j} \ot 1) \Na_{n - 1} S_{x_j}^* & \T{for} & n \geq 1 \\
\sum_{i = 1}^l (S_{y_i}^* \ot 1) \Na_{n+1} S_{y_i} & \T{for} & n \leq -1 .}
\end{equation}

To explain how the Gra\ss mann connections relate to the {\it left action} of $\C B_\de$ on $\C E_n$, $n \in \zz$, we introduce the $*$-homomorphisms
\[
\al_n : \C B_\de \to \fork{cc}{M_{k^{n}}(\C B_{\delta}) \quad & n\geq 0\\
M_{l^{|n|}}(\C B_{\delta}) \quad & n<0, } \quad\quad\quad\quad b\mapsto v_{n}\pi_{n}(b)v_{n}^{*},
\]
using the $\C B_\de$-linear maps from \eqref{eq:restrict}. It can then be verified that
\begin{equation}\label{eq:leftaction}
\Na_n(b \cd \xi) = b \cd \Na_n(\xi) + v_n^* \de( \al_n(b) ) v_n(\xi) \q b \in \C B_\de \, , \, \, \xi \in \C E_n \, .
\end{equation}

We define the dense submodule
\[
\C X := \bigoplus_{n \in \zz} \C E_n \su X,
\]
together with the closable hermitian $\de$-connection 
\[
\Na_\de : \C X \to X \hot_B \Om(\C B_\de), \q \Na_\de( ( \xi_n)_{n\in\mathbb{Z}} ) := (\Na_n(\xi_n))_{n\in\mathbb{Z}}.
\]

We consider the associated closed $*$-derivation $\delta_{\nabla}:\T{Dom}(\delta_{\nabla}) \to \mbox{End}^{*}_{\Omega(\C B_{\de})}\big(X\hot_{B}\Omega(\C B_{\de}) \big)$ constructed in Subsection \ref{Ex: Hilher}. The results in Subsection \ref{Ex: Hilher} then shows that we obtain an operator $*$-correspondence $\C X_\Na$ from $\T{Lip}(\de_\Na)$ to $\C B_\de$. We now address the question whether $\delta_{\nabla}$ is densely defined in $\C O_{E}$, that is, whether the intersection $\C O_E \cap \T{Lip}(\de_\Na)$ is dense in $\C O_E$. 
%
\medskip

For each $\xi \in \C E$, it is clear that the bounded adjointable operators $S_{\xi} \, , \, \, S_{\xi}^* : X\to X$ defined in \eqref{creabea} restrict to $\C B_\de$-linear maps $S_{\xi} \, , \, \, S_{\xi}^* : \C X \to \C X$. Applying \eqref{eq:relgrass} and \eqref{eq:leftaction} we may compute the commutator of these maps with the hermitian $\de$-connection $\nabla_{\de} : \C X \to X \hot_B \Om(\C B_\de)$. We state the result as a lemma:

\begin{lemma}\label{comm}
Let $\xi \in \C E$. 
We have the explicit formulae
\[
\Na_{n + 1} S_\xi - (S_\xi \ot 1) \Na_n = \fork{ccc}{
\sum_{j = 1}^k (S_{x_{j}} \ot 1) v_n^*  \cd \de\big( \alpha_n (\inn{x_{j},\xi} ) \big) \cd v_n
& \T{for} & n\geq 0 \\
\sum_{i=1}^l v_{n+1}^{*}\cdot \delta\big( \alpha_{n+1}(\langle \xi^{*},y_{i}^{*} \rangle)\big) \cdot v_{n+1}S_{y_{i}} & \T{for}& n \leq -1,
}
\]
and
\[
\Na_n S_\xi^* - (S_\xi^* \ot 1) \Na_{n+1} = \fork{ccc}{
\sum_{j = 1}^k v_n^* \cd \de\big( \al_n( \inn{\xi,x_j}) \big) \cd v_n S_{x_j}^* & \T{for} & n \geq 0 \\
\sum_{i = 1}^l (S_{y_i}^* \ot 1) v_{n+1}^* \cd \de\big( \al_{n+1}(\inn{y_i^*,\xi^*}) \big) \cd v_{n+1} & \T{for} & n \leq -1,
}
\]
for the commutator between the hermitian $\de$-connection $\nabla_{\delta}$ and the creation and annihilation operators associated to $\xi \in \C E$.
\end{lemma}
\begin{proof} As mentioned above, we apply \eqref{eq:relgrass} and \eqref{eq:leftaction}. For $n \geq 0$ we find that
\[
\begin{split}
\Na_{n + 1} S_\xi - ( S_\xi \ot 1) \Na_n
& = \sum_{j = 1}^k (S_{x_j} \ot 1) \Na_n S_{x_j}^* S_\xi - (S_\xi \ot 1) \Na_n \\
& = \sum_{j = 1}^k (S_{x_j} S_{x_j}^* S_\xi \ot 1) \Na_n - (S_\xi \ot 1) \Na_n
+ (S_\xi \ot 1)\sum_{j = 1}^k v_n^* \de\big( \al_n( \inn{x_j,\xi}) \big) v_n \\
& = \sum_{j = 1}^k (S_{x_j} \ot 1)v_n^* \de\big( \al_n( \inn{x_j,\xi}) \big) v_n.
\end{split}
\]

For $n \leq - 1$ we find that
\[
\begin{split}
\Na_{n + 1} S_\xi - ( S_\xi \ot 1) \Na_n
& = \Na_{n + 1} S_\xi - \sum_{i = 1}^l (S_\xi S_{y_i}^* \ot 1) \Na_{n + 1} S_{y_i} \\
& = \Na_{n + 1} S_\xi - \sum_{i = 1}^l \Na_{n + 1} S_\xi S_{y_i}^* S_{y_i}
+ \sum_{i = 1}^l v_{n + 1}^* \de\big( \al_{n + 1}( \inn{\xi^*,y_i^*}) \big) v_{n + 1} S_{y_i} \\
& = \sum_{i = 1}^l v_{n + 1}^* \de\big( \al_{n + 1}( \inn{\xi^*,y_i^*}) \big) v_{n + 1} S_{y_i}.
\end{split}
\]

A similar computation can be used to verify the formulae for the annihilation operator $S_{\xi}^*$.

\end{proof}

\begin{prop}\label{prop: eqder} The following are equivalent:
\begin{enumerate}
\item $\sup_{n \in \zz}\| v_{n}^{*}\delta(\alpha_{n}(b))v_{n}\|_\infty <\infty$ for all $b\in \C B_{\delta}$;
\item $S_{\xi} \in \T{Dom}(\delta_{\nabla})$ for all $\xi\in\mathcal{E}$;
\item $\C B_{\delta}\su \T{Dom}(\delta_{\nabla})$.  
\end{enumerate}
\end{prop}
\begin{proof} $1.\Rightarrow 2.$: By Lemma \ref{comm}, the operators
$$\de_{\Na}^{0}(S_{\xi}) \, , \, \, \de_{\Na}^{0}(S_{\xi}^*) : \C X_\Na\ot_{\C B_{\de}}^{\textnormal{alg}}\Om(\C B_\de)\to X \hot_B \Om(\C B_\de), $$
extend boundedly to $X \hot_B \Om(\C B_\de)$ for all $\xi\in \mathcal{E}$ whenever $\sup_{n \in \zz}\|v_{n}^{*}\delta(\alpha_{n}(b))v_{n}\|_\infty <\infty$. 

$2.\Rightarrow 3.$:  For all $b \in \C B_\de \su \C O_E$ we have that
\[
b = \sum_{j=1}^k S_{bx_{j}}S_{x_{j}}^{*} \in \T{Dom}(\delta_{\nabla}).
\]

$3.\Rightarrow 1.$ Since $\C B_{\delta}\su \T{Dom}(\delta_{\nabla})$ it follows that
\[
[\Na_\de, b] = \bigoplus_{n\in\mathbb{Z}} v_{n}^{*}\delta\big(\alpha_{n}(b)\big)v_{n},
\]
is a bounded adjointable operator for all $b \in \C B_\de$. The norm of this diagonal operator is $\sup_{n \in \zz}\|v_{n}^{*}\delta(\alpha_{n}(b))v_{n}\|_\infty$.
\end{proof}

Motivated by the above proposition we make the following additional:

\begin{assu}
It is assumed that
\[
\sup_{n \in \zz}\| v_{n}^{*}\delta(\alpha_{n}(b))v_{n}\|_\infty <\infty,
\]
for all $b\in \C B_{\delta}$
\end{assu}
A similar assumption first appeared in the context of spectral triples for crossed products by the integers in \cite{BMR} and later in \cite[Proposition 4.8]{GMR} in the context of Cuntz-Pimsner algebras associated to vector bundles. It should be compared to Assumption \ref{a:grpact}.2 of the present paper.

We define the operator $*$-algebra $\mathcal{O}_{\mathcal{E}}$ to be the closure of the $*$-algebra generated by the operators $S_{\xi}$ with $\xi\in \mathcal{E}$ inside the operator $*$-algebra $\textnormal{Lip}(\delta_{\nabla})$, see Subsection \ref{Ex: Hilher}. The results of Subsection \ref{Ex: Hilher} now provides $\C X_{\nabla}$ with the structure of an operator $*$-correspondence from $\mathcal{O}_{\mathcal{E}}$ to $\C B_{\delta}$.

\section{The standard form of operator $*$-correspondences}\label{s:charcor}
We now prove a representation theorem for operator $*$-correspondences. Like operator $*$-algebras, operator $*$-correspondences always admit a so called \emph{standard form representation}. As a consequence, we deduce that, like Hilbert $C^{*}$-modules over $C^{*}$-algebras, operator $*$-modules admit a \emph{linking operator $*$-algebra}. Up to cb-isomorphism, any operator $*$-module arises as a corner in an operator $*$-algebra. The examples in Section 3 can easily be put in standard form.

\subsection{Implementing the inner product of a concrete operator $*$-correspondence}\label{s:conclink}

Let $\C X$ be an operator $*$-correspondence over a pair of operator $*$-algebras $(\C A,\C B)$. The following assumption will remain in effect throughout this subsection:

\begin{assu}\label{a:faithCES}
We assume that we may find a Hilbert space $H_{0}$ together with completely isometric algebra homomorphisms $\ga_{\C A} : \C A \to B(H_{0})$ and $\ga_{\C B} : \C B \to B(H_{0})$ together with a complete isometry $\ga_{\C X} : \C X \to B(H_{0})$ such that the relations
\[
\ga_{\C A}(a) \cd \ga_{\C X}(x) = \ga_{\C X}(a \cd x) \q \T{and} \q
\ga_{\C X}(x) \cd \ga_{\C B}(b) = \ga_{\C X}(x\cd b),
\]
hold for all $a \in \C A$, $x \in \C X$ and $b \in \C B$.
\end{assu}

From the data in Assumption \ref{a:faithCES} we consider the Hilbert space $H:=H_{0}\oplus H_{0}$ and the completely isometric algebra homomorphisms $\io_{\C A}: \C A \to B(H) $ and  $\io_{\C B} : \C B \to B(H)$ given by 

\begin{align*} \io_{\C A}(a) := \begin{pmatrix} \ga_{\C A}(a) & 0 \\ 0 & \ga_{\C A}(a^\da)^*\end{pmatrix}, \q \io_{\C B}(b) &:= \begin{pmatrix} \ga_{\C B}(b^{\dag})^*& 0 \\ 0 &  \ga_{\C B}(b)\end{pmatrix}. \end{align*}
 and the complete isometries $ \io_{\C X} : \C X \to B(H)$ and $\io_{\C X^\da} : \C X^\da \to B(H)$ given by 
 \begin{equation}\label{eq:iotax}
 \io_{\C X}(x) :=\begin{pmatrix}0 & \ga_{\C X}(x) \\ 0 & 0\end{pmatrix},\q  \io_{\C X^\da}(x^\da) :=\begin{pmatrix}0 & \ga_{\C X}(x)^{*} \\ 0 & 0\end{pmatrix}.
\end{equation}
The self-adjoint unitary $U := \ma{cc}{0 & 1 \\ 1 & 0}$ implements the relations
\begin{equation}\label{eq:symform}\iota_{\C A}(a^{\dag})=U\iota_{\C A}(a)^{*}U,\q \iota_{\C B}(b^{\dag})=U\iota_{\C B}(b)^{*}U,\q \iota_{\C X^{\dag}}(x^{\dag})=U\iota_{\C X}(x)^{*}U.\end{equation}
Moreover, we have the relations:

\begin{align}
\label{eq:Bmod}&  \io_{\C X}(x) \cd \io_{\C B}(b) = \io_{\C X}( x b  ),\q \io_{\C B}(b) \cd \io_{\C X^\da}(x^\da) = \io_{\C X^\da}( (x b^\da )^\da ),\\
\label{eq:Amod} &  \io_{\C A}(a)\cd \io_{\C X}(x )  = \io_{\C X}( a x ) \q \io_{\C X^\da}(x^\da ) \cd \io_{\C A}(a) = \io_{\C X^\da}( (a^\da x)^\da ) \q \T{and} \\
\label{eq: nilpotent}&\io_{\C X}(x)\io_{\C X}(y)=\io_{\C X^{\dag}}(y^{\dag})\io_{\C X}(x)=\io_{\C X}(x)\io_{\C X^{\dag}}(y^{\dag})=\io_{\C X^{\dag}}(y^{\dag})\io_{\C X^{\dag}}(x^{\dag})=0.
\end{align}

\medskip

We define the completely contractive bilinear map
\begin{equation}\label{eq:bilin}
 \phi : \C X^\da  \ti \C X \to B(H), \q \phi( x^\da, y) = \io_{\C B}( \inn{x, y}_{\C X} ) .
\end{equation}
By \cite[Theorem 3.2]{PaSm:MTO} and \cite[Corollary 3.2]{ChEfSi:CBM} there exist a Hilbert space $K$, and  $*$-homomorphisms
\[
\si_{\C X^\da }\, , \, \, \si_{\C X} : B(H) \to B(K) ,
\]
together with contractions $T : H \to K$ and $S : K \to H$ such that
\begin{equation}\label{eq:defining}
S \cd (\si_{\C X^\da} \ci \io_{\C X^\da})(x^\da) \cdot (\si_{\C X} \ci \io_{\C X})(y) \cd T = \phi(x^\da, y) ,
\end{equation}
for all $x,y \in \C X$. We define the completely contractive algebra homomorphism
\[ 
\rho_{\C A} : \C A \to B(K) \q \rho_{\C A}(a):=  \si_{\C X^\da}(\iota_{\C A}(a)), 
\]
as well as the completely contractive maps
\[
\rho_{\C X} := (\si_{\C X} \ci \io_{\C X}) : \C X \to B(K) \, , \, \, 
\rho_{\C X^\da} : (\si_{\C X^\da} \ci \io_{\C X^\da}) : \C X^\da \to B(K) \, .
\]

\begin{remark}\label{warning} It holds that 
\[
\rho_{\C X^\da}(x^\da)\rho_{\C A}(a)=\si_{\C X^{\da}}(\io_{\C X^\da}(x^\da)\io_{\C A}(a)) = \rho_{\C X^\da}((a^\da x)^\da),
\]
for all $x \in \C X$ and $a \in \C A$. Consequently, we have that
\[
S \cd \rho_{\C X^\da}(x^\da) \cdot \rho_{\C A}(a)\cdot \rho_{\C X}(y) \cd T = \phi((a^{\da} x)^{\da},y) = \phi(x^{\da}, ay),
\]
for all $x,y \in \C X$, $a \in \C A$. However, it need \emph{not} hold that $\rho_{\C A}(a)\rho_{\C X}(x)=\rho_{\C X}(ax)$.

\end{remark}

Define the subspace
\[
E :=\ov{\rho_{\C X^\da}(\C X^\da)^* \cd S^* H } \su K,
\]
and let $P\in B(K)$ be the orthogonal projection with $\T{Im}(P) = E$. Then define the subspace
\[ 
F :=\ov{ P\cd \rho_{\C X}(\C X) \cd T H} \su K, 
\]
and let $Q\in B(K)$ be the orthogonal projection with $\T{Im}(Q) = F$. In what follows, we will often remove the subscripts from the completely contractive maps $\rho_{\C A}$, $\rho_{\C X}$ and $\rho_{\C X^\da}$.

\begin{lemma}\label{l:relations} For $x\in \C X$ and $a\in \C A$ we have the relations
\begin{align}
\label{rel1}&PQ=QP=Q,\quad P\rho(a)P=P\rho(a), \quad Q\rho(a)P=Q\rho(a) \, , \\  
\label{rel3}& Q\rho(x)T=P\rho(x)T , \q S\rho(x^\da)P=S\rho(x^{\da}) \, , \\
\label{relnew} &P\rho(a)P\rho(x)T=P\rho(ax)T \, , \\ 
\label{rel2}&Q\rho(a)Q=P\rho(a)Q\, .  
\end{align}

\end{lemma}
\begin{proof}
The relation $PQ=Q=QP$ is immediate from the definition of $Q$. The linear subspace $\rho(\C A)^{*}\su B(K)$ is a subalgebra and by Remark \ref{warning} we have that $\rho(\C A)^{*}E\su E$. Therefore it holds that  $P\rho(a)^{*}P=\rho(a)^{*}P$ and consequently $P\rho(a)P=P\rho(a)$ for all $a\in \C A$. The fact that $Q \rho(a) P = Q \rho(a)$ now follows by multiplying the identity $P \rho(a) P = P \rho(a)$ from the left by $Q$. This proves Relation \eqref{rel1}. The relations in \eqref{rel3} follow immediately from the definition of the orthogonal projections $P$ and $Q \in B(K)$.

To prove the relation in \eqref{relnew}, we let $y \in \C X$ and $\xi,\eta \in H$. Using Relation \eqref{rel1}, the definition of $P$ and the defining relations \eqref{eq:defining}, we find that 
\begin{align*}
\binn{P\rho(a)P\rho(x)T\xi,\rho(y^\da)^{*}S^{*}\eta}
&=\binn{P\rho(a)\rho(x)T\xi,\rho(y^\da)^{*}S^{*}\eta}
=\binn{\rho(a)\rho(x)T\xi,\rho(y^\da)^{*}S^{*}\eta}\\
&=\binn{ \io_{\C B}( \inn { y, a x}_{\C X} ) \xi,\eta}=\binn{S\rho(y^{\da})\rho(ax)T\xi,\eta} 
= \binn{P\rho(ax)T\xi,\rho(y^\da)^{*}S^{*}\eta}\, .
\end{align*}
By the definition of the orthogonal projection $P \in B(K)$ we thus have that $$\inn{P \rho(a) P \rho(x) T \xi,\ze} = \binn{P\rho(ax)T\xi,\ze},$$ for all $\xi \in H$, $\ze \in K$ and this implies Relation \eqref{relnew}. 

To prove the relation in \eqref{rel2} we let $x \in \C X$ and $\xi \in H$ be given. Using Relation \eqref{relnew} we then have that
\begin{align*}Q \rho(a) Q P\rho(x)T\xi &=QP\rho(a)P\rho(x)T\xi=QP\rho(ax)T\xi \\&=P\rho(ax)T\xi =P\rho(a)P\rho(x)T\xi=P\rho(a)QP\rho(x)T\xi \, ,\end{align*}
Since the vectors of the form $P\rho(x)T\xi$ span a dense subspace of the image of $Q : K \to K$, it follows that $Q\rho(a)Q=P\rho(a)Q$ for all $a \in \C A$. \end{proof}
On the modules $\C X$ and $\C X^\da$ we define the completely bounded maps: 
\begin{align}\label{gam1}
\te_{\C X} & := \ma{cc}{ 0 & Q \cd \rho_{\C X} \cd T \\ 0 & \io_{\C X} } : \C X \to B(K \op H), \\
\te_{\C X^\da} &:= \ma{cc}{0 & 0 \\ S \cd \rho_{\C X^\da} \cd Q & \io_{\C X^{\dag}}} : \C X^\da \to B(K \op H),
\end{align}
and on the algebras $\C A$ and $\C B$ we define the complete isometries:
\begin{align}\label{gam2}
\te_{\C A} & := \ma{cc}{ Q\cdot\rho_{\C A}\cdot Q & 0 \\ 0 & \io_{\C A} } : \C A \to B(K \op H), \\
\te_{\C B} & := \ma{cc}{0 & 0 \\ 0 & \io_{\C B}} : \C B \to B(K \op H) .
\end{align}
\begin{lemma}\label{l:images}
We have the identities of subspaces:
\[ 
\ov{ Q \rho(\C X) T H} = Q K = \ov{ Q \rho(\C X^\da)^* S^* H} ,
\]
and in particular 
\[
\ov{ (Q \oplus 0_H) \te_{\C X}(\C X)(K\oplus H)} = (Q \op 0_H)(K \op H) 
= \ov{(Q \oplus 0_H) \te_{\C X^\da}(\C X^\da)^{*}(K\oplus H)}\, .
\] 

\end{lemma}
\begin{proof} We observe that
\[
\ov{Q\rho(\C X)TH}=\ov{QP\rho(\C X)TH}=\ov{P\rho(\C X)TH}=QK,
\]
and similarly that
\[
\ov{Q\rho(\C X^{\da})^{*}S^{*}H}=QPK=QK \, ,
\]
which completes the proof of the first two identities. The second pair of identities is an immediate consequence of the first pair of identities.
\end{proof}

\begin{prop}\label{p:stine}
The maps $\te_{\C A} : \C A\to B(K \op H )$ and $\te_{\C B} : \C B \to B(K \op H )$ are completely isometric algebra homomorphisms and the linear maps $\te_{\C X} : \C X \to B(K \op H )$ and $\te_{\C X^\da} : \C X^\da \to B(K \op H )$ satisfy the estimates:
\[
\| x \| \leq \| \te_{\C X}( x) \| \leq \sqrt{2} \cd \| x \| \q \T{and} \q
\| x^\da \| \leq \| \te_{\C X^\da}(x^\da) \| \leq \sqrt{2} \cd \| x^\da \|,
\]
for all $x \in M_m(\C X)$, $m \in \nn$. Moreover, we have the relations:
\begin{align}
\label{impl1}& \te_{\C X^\da}(x^\da) \cd \te_{\C X}(y) = \te_{\C B}( \inn{x,y}_{\C X} ) ,\\
\label{impl2}& \te_{\C A}(a) \cd \te_{\C X}(y) = \te_{\C X}(a \cd y  ), \q   \te_{\C X^\da}(x^\da)\cd \te_{\C A}(a) = \te_{\C X^\da}( (a^\da x )^\da ) ,  \\
\label{impl3}& \te_{\C X}(y) \cd \te_{\C B}(b) = \te_{\C X}(y \cd b), \q  \te_{\C B}(b) \cd \te_{\C X^\da}(x^\da) = \te_{\C X^\da}( ( x\cd b^\da )^\da ),
\end{align}

for all $x,y \in \C X$, $a \in \C A$ and $b \in \C B$. 

\end{prop}

\begin{proof}
It is clear that the maps $\te_{\C A}$ and $\te_{\C B}$ are completely isometric and moreover that $\te_{\C B}$ is an algebra homomorphism. The norm-estimates for $\te_{\C X}$ and $\theta_{\C X^{\dag}}$ follow since $\io_{\C X}$ and $\iota_{\C X^{\dag}}$ are complete isometries and  $Q \cd \rho_{\C X} \cd T$ and $S\cd\rho_{\C X^{\dag}}\cd Q$ are completely contractive. 
\medskip

The fact that $\te_{\C A}$ is an algebra homomorphism follows from Lemma \ref{l:relations}:
\[
\te_{\C A}(a)\te_{\C A}(b)=Q\rho(a)Q\rho(b) Q=Q\rho(a)P\rho(b)Q= Q\rho(a)\rho(b)Q= Q \rho(ab) Q=\te_{\C A}(ab) \,.
\]

We now verify the identity in \eqref{impl1}. We apply the relations in Lemma \ref{l:relations} and Equations \eqref{eq: nilpotent}  and \eqref{eq:defining}. Indeed, we compute that
\[
\te_{\C X^\da}(x^\da) \cd \te_{\C X}(y) 
= 0 \op S \rho(x^\da) Q \rho( y) T = 0 \op S \rho(x^\da) \rho( y) T
= \te_{\C B}( \inn{x, y}_{\C X}) \, .
\]

Next we consider the left module structure applying Lemma \ref{l:relations} and Equation \eqref{eq:Amod}:
\[
\te_{\C A}(a) \cd \te_{\C X}(y)=
\ma{cc}{ 0 & Q\rho(a)Q\rho(y)T \\ 
0 & \io_{\C A}(a)\io_{\C X}(y) } = \ma{cc}{0 & Q\rho(a)P\rho(y)T \\ 0 & \io_{\C X}(a y) } = \te_{\C X}(ay) \, .
\]
Similarly, we obtain the second part of \eqref{impl2} from Lemma \ref{l:relations}, Remark \ref{warning} and Equation \eqref{eq:Amod}. Indeed,
\[
\begin{split}
\te_{\C X^\da}(x^\da)\cd \te_{\C A}(a)
& = \ma{cc}{0 & 0 \\
S\rho(x^\da)Q\rho(a)Q & \io_{\C X^{\dag}}(x^\da)\io_{\C A}(a)} 
= \ma{cc}{0 & 0 \\ 
S\rho(x^\da)P\rho(a)Q & \io_{\C X^{\dag}}((a^{\dag}x)^{\dag}) }  \\
& = \ma{cc}{ 0 & 0 \\
S\rho(x^\da)\rho(a)Q & \io_{\C X^{\dag}}((a^{\dag}x)^{\dag}) } 
= \te_{\C X^\da}( (a^\da x)^\da) \, .
\end{split}
\]

We now consider the $\C B$-module structures, starting with the first identity in \eqref{impl3}. By definition of $\te_{\C B}$ and $\theta_{\C X}$ and Equation \eqref{eq:Bmod} it is immediate that
\[
\theta_{\C X}(y)\cd\theta_{\C B}(b)-\theta_{\C X}(y b) 
= (Q\oplus 0_H)(\theta_{\C X}(y)\cd \theta_{\C B}(b)-\theta_{\C X}(y b)) \, .
\]
By Lemma \ref{l:images} it thus suffices to show that for all $\xi\in K\oplus H$ and $\eta \in \ov{(Q\oplus 0_H) \theta_{\C X^{\dag}}(\C X^{\dag})^{*}(K\oplus H)}$ it holds that
\begin{equation*}
\binn{ (Q \oplus 0_H)(\theta_{\C X}(y)\cd\theta_{\C B}(b)-\theta_{\C X}(yb))\xi, \eta } = 0 \, . 
\end{equation*}
This in turn follows, if for all $x\in \C X$ we have 
\begin{equation}\label{eq:orth}
\theta_{\C X^{\dag}}(x^{\dag})(\theta_{\C X}(y)\theta_{\C B}(b)-\theta_{\C X}(yb))=\theta_{\C X^{\dag}}(x^{\dag})\cd (\theta_{\C X}(y)\theta_{\C B}(b)-\theta_{\C X}(yb))=0 \, .
\end{equation}
To this end we compute, using \eqref{impl1}, that
\[
\te_{\C X^\da}(x^\da) \cd \te_{\C X}(y) \cd \te_{\C B}(b) 
= \te_{\C B}( \inn{x,y}_{\C X} \cd b) 
= \te_{\C X^\da}(x^{\da})\cd \te_{\C X}(y\cd b) \, ,
\]
which yields \eqref{eq:orth}. The second identity in \eqref{impl3} is proved in a similar manner by observing that
\[
\te_{\C B}(b)\cd \te_{\C X^{\dag}}(x^{\dag})-\te_{\C X^{\dag}}((x\cd b^{\dag})^{\dag})=( \te_{\C B}(b)\cd \te_{\C X^{\dag}}(x^{\dag})-\te_{\C X^{\dag}}((x\cd b^{\dag})^{\dag}))(Q\oplus 0_H) \, .
\]
Indeed, employing Lemma \ref{l:images}, it suffices that the above expression vanishes on all elements in the subspace $\ov{ (Q\oplus 0_H)\te_{\C X}(\C X)(K\oplus H) } \su K \op H$. This is established since
\[
(\te_{\C B}(b)\cd \te_{\C X^{\dag}}(x^{\dag})-\te_{\C X^{\dag}}((x\cd b^{\dag})^{\dag})\te_{\C X}(y)=0, \q  \forall y \in \C X \, ,
\]
which completes the proof.
\end{proof}

We now arrive at the main result of this section.
\begin{thm}\label{t:isochar}
Let $\C X$ be an operator $*$-correspondence over $(\C A,\C B)$ satisfying Assumption \ref{a:faithCES}. Then there exist a Hilbert space $H_{\pi}$ together with completely isometric algebra homomorphisms $\pi_{\C B} : \C B \to B(H_{\pi})$ and $\pi_{\C A} : \C A \to B(H_{\pi})$, a completely bounded linear map $\pi_{\C X} : \C X \to B(H_{\pi})$, and a selfadjoint unitary $U \in B(H_{\pi})$ such that
\begin{enumerate}
\item $\pi_{\C X}$ satisfies the norm-estimate:
\[
\| x \|_{\C X} \leq \| \pi_{\C X}(x) \| \leq \sqrt{2} \cd \| x \|_{\C X},
\]
for all $x \in M_n(\C X)$, $n \in \nn$. In particular $\C X$ is cb-isomorphic to the image $\pi_{\C X}(\C X) \su  B(H_\pi)$.
\item The module relations hold:
\[
\pi_{\C X}(x \cd b) = \pi_{\C X}(x) \cd \pi_{\C B}(b), \q
\pi_{\C A}(a) \cd \pi_{\C X}(x) = \pi_{\C X}(a \cd x) ,
\]
for all $x \in \C X$ and all $a \in \C A$, $b \in \C B$.
\item $U$ implements the involutions: 
\[
U \pi_{\C A}(a)^* U = \pi_{\C A}(a^\da), \q U \pi_{\C B}(b)^* U = \pi_{\C B}(b^\da),
\]
for all $a \in \C A$, $b \in \C B$.
\item $U$ implements the inner product:
\[
U \pi_{\C X}(x)^* U \pi_{\C X}(y) = \pi_{\C B}( \inn{x,y} ),
\]
for all $x,y \in \C X$.
\end{enumerate}
\end{thm}
\begin{proof}
We choose the Hilbert spaces $H,K$ and the maps $\te_{\C A}$, $\te_{\C B}$, $\te_{\C X}$, $\te_{\C X^\da}$ as in Proposition \ref{p:stine}. We then define the Hilbert space $H_{\pi}:=(K\oplus H)^{\oplus 2}$ and the  selfadjoint unitary operator
\[
U:=\begin{pmatrix}0 & 1_{K\oplus H}  \\  1_{K\oplus H}& 0 \end{pmatrix} \in B( H_\pi ) \, .
\]
Subsequently, we define the completely bounded linear map
\[
\pi_{\C X} : \C X \to B(H_{\pi}), \q \pi_{\C X}(x) := \ma{cc}{\te_{\C X}(x) & 0 \\ 0 & \te_{ \C X^\da}(x^\da)^{*} },
\]
and the completely isometric algebra homomorphisms
\[
\begin{split}
& \pi_{\C B} : \C B \to B(H_{\pi}), \q \pi_{\C B}(b) := \ma{cc}{\te_{ \C B}(b) & 0 \\ 0 & \te_{\C B}(b^{\da})^{*}} \q \T{and} \\
& \pi_{\C A} : \C A \to B(H_{\pi}), \q \pi_{\C A}(a) :=  \ma{cc}{\te_{ \C A}(a) & 0 \\ 0 & \te_{ \C A}(a^\da)^*} \, .
\end{split}
\]
The fact that these maps satisfy the identities $(2)$, $(3)$ and $(4)$ and the required norm estimates can be verified by straightforward methods using Proposition \ref{p:stine}.
\end{proof}

\subsection{The standard form of an operator $*$-correspondence} Inspired by Theorem \ref{t:isochar} we make the following general definition. We emphasize that in the present section Assumption \ref{a:faithCES} is no longer in effect.

\begin{dfn}\label{d:staform}
Let $\C X$ be an operator $*$-correspondence from $\C A$ to $\C B$. We say that a quintuple $(H_\pi, U, \pi_{\C A}, \pi_{\C B}, \pi_{\C X})$ is a \emph{standard form representation of $\C X$}, when $H_\pi$ is a Hilbert space, $U \in B(H_\pi)$ is a selfadjoint unitary operator, $\pi_{\C A} : \C A \to B(H_\pi)$ and $\pi_{\C B} : \C B \to B(H_\pi)$ are completely bounded algebra homomorphisms and $\pi_{\C X} : \C X \to B(H_\pi)$ is a completely bounded linear map such that
\begin{enumerate}
\item $\C A$, $\C B$ and $\C X$ are cb-isomorphic to their respective images $\pi_{\C A}(\C A)$, $\pi_{\C B}(\C B)$ and $\pi_{\C X}(\C X) \su B(H_\pi)$ via the maps $\pi_{\C A}, \pi_{\C B}$ and $\pi_{\C X}$ (in particular these images are closed in operator norm).
\item We have the module relations:
\[
\pi_{\C X}(x \cd b) = \pi_{\C X}(x) \cd \pi_{\C B}(b), \q \pi_{\C A}(a \cd x) = \pi_{\C A}(a) \cd \pi_{\C X}(x),
\]
for all $a \in \C A$, $x \in \C X$, $b \in \C B$.
\item The selfadjoint unitary $U \in B(H_\pi)$ implements the involutions and the inner product:
\[
U \pi_{\C A}(a)^* U = \pi_{\C A}(a^\da) \, \, , \q U \pi_{\C B}(b)^* U = \pi_{\C B}(b^\da) 
\, \, , \q U \pi_{\C X}(x)^* U \pi_{\C X}(y) = \pi_{\C B}( \inn{x,y}),
\]
for all $a \in \C A$, $b \in \C B$, $x,y \in \C X$ (where $*$ refers to the Hilbert space adjoint operation).
\end{enumerate}
\end{dfn} 
Having established the existence of standard form representations for operator $*$-correspondences satisfying Assumption \ref{a:faithCES}, we are now ready to prove our main result: The existence of a standard form representation for an arbitrary operator $*$-correspondence. 
\newline

Let $\C A$ and $\C B$ be operator $*$-algebras and let $\C X$ be an operator $*$-correspondence from $\C A$ to $\C B$. Since $\C X$ is, in particular, an operator $\C A$-$\C B$-bimodule, by \cite[Theorem 5.2.17]{BlMe:OMO} there exist a Hilbert space $H_{0}$, completely bounded algebra homomorphisms $\al_{\C A} : \C A \to B(H_{0})$ and $\al_{\C B} : \C B \to B(H_{0})$ and a completely bounded linear map $\al_{\C X} : \C X \to B(H_{0})$ such that
\begin{enumerate}
\item $\C A$, $\C B$ and $\C X$ are cb-isomorphic to their respective images $\al_{\C A}(\C A)$, $\al_{\C B}(\C B)$ and $\al_{\C X}(\C X) \su B(H_0)$ via the above maps (in particular, these images are closed in operator norm).
\item The module relations hold:
\[
\al_{\C A}(a) \cd \al_{\C X}(x) = \al_{\C X}(a \cd x) \q \T{and} \q
\al_{\C X}(x) \cd \al_{\C B}(b) = \al_{\C X}(x \cd b),
\]
for all $x \in \C X$, $a \in \C A$, $b \in \C B$.
\end{enumerate}

Using the maps $\al_{\C A}$, $\al_{\C B}$ and $\al_{\C X}$, we introduce new matrix norms on $\C A$, $\C B$ and $\C X$ and denote the resulting operator spaces by $\C A_{\al}$, $\C B_{\al}$ and $\C X_{\al}$, respectively. Let $H:=H_{0}\oplus H_{0}$. The matrix norms on $\C A_{\al} = \C A$ and $\C B_{\al} = \C B$ are defined via the injective algebra homomorphisms
\begin{equation}\label{eq:alpisometI}
\begin{split}
& \ga_{\C A_\al} : \C A \to B(H), \q \ga_{\C A_\al} : a\mapsto \begin{pmatrix}\al_{\C A}(a) & 0 \\ 0 & \al_{\C A}(a^{\da})^{*} \end{pmatrix}, \\
& \ga_{\C B_\al} : \C B \to B(H), \q \ga_{\C B_\al} : b\mapsto \begin{pmatrix}\al_{\C B}(b^\da)^{*} & 0 \\ 0 & \al_{\C B}(b) \end{pmatrix} 
\, ,
\end{split}
\end{equation}
making $\C A_{\al}$ and $\C B_{\al}$ into operator $*$-algebras  (where the $*$-algebra structures come from the $*$-algebra structures on $\C A$ and $\C B$, respectively). The matrix norms on $\C X_{\al} = \C X$ are defined via the injective linear map 
\begin{equation}\label{eq:alpisometII}
\ga_{\C X_\al} : \C X \to B(H), \q \ga_{\C X_\al}: x\mapsto \begin{pmatrix} 0 & \al_{\C X}(x) \\ 0 & 0\end{pmatrix}  \, . 
\end{equation}
 Clearly $\C X_\al$ becomes an operator $\C A_\al$-$\C B_\al$-bimodule when equipped with the bimodule structure coming from $\C X$. Moreover, we have the following:

\begin{lemma}
Let $C := \| \al_{\C B} \|_{\T{cb}} \cd \| \al_{\C X}^{-1} \|_{\T{cb}}^2$. The inner product $\inn{\cd,\cd}_{\C X_\al} : \C X_\al \ti \C X_\al \to \C B_\al$ defined by $\inn{x,y}_{\C X_\al} := \frac{1}{C} \cd \inn{x,y}_{\C X}$, $x,y \in \C X_\al$, provides $\C X_\al$ with the structure of an operator $*$-correspondence from $\C A_\al$ to $\C B_\al$.
\end{lemma}
\begin{proof}
We only address the Cauchy-Schwarz inequality. This follows from the estimates
\[
\begin{split}
& \| \inn{x,y}_{\C X_\al} \|_{\C B_\al} = \frac{1}{C} \| \ga_{\C B_\al}(\inn{x,y}_{\C X}) \|_\infty \\
& \q \leq \frac{1}{C} \cd \sup \{ \| \al_{\C B}( \inn{y,x}_{\C X} ) \|_\infty \, , \, \, \| \al_{\C B}( \inn{x,y}_{\C X} ) \|_\infty \} \\
& \q \leq \frac{1}{C} \cd \| \al_{\C B} \|_{\T{cb}} \cd \| \inn{x,y}_{\C X} \|_{\C B}
\leq \frac{1}{C} \cd \| \al_{\C B} \|_{\T{cb}} \cd \| x \|_{\C X} \cd \| y \|_{\C X} \\
& \q \leq \frac{1}{C} \cd \| \al_{\C B} \|_{\T{cb}} \cd \| \al_{\C X}^{-1} \|_{\T{cb}}^2 \cd \| \al_{\C X}(x) \|_\infty \cd \| \al_{\C X}(y) \|_\infty
= \| x \|_{\C X_\al} \cd \| y \|_{\C X_\al} \, ,
\end{split}
\]
which are valid for all $x,y \in M_n(\C X_\al)$.
\end{proof}

 Clearly, the two operator $*$-correspondences $\C X$ and $\C X_\al$ are cb-isomorphic in the sense of Definition \ref{d:opstarcor}. The cb-isomorphism is provided by the cb-isomorphisms $\C A \to \C A_\al$ and $\C B \to \C B_\al$ (both induced by the identity map) as well as the cb-isomorphism $\C X \to \C X_\al$ given by multiplication by $\frac{1}{\sqrt{C}}$, where the constant $C > 0$ is defined in the above lemma. Moreover, we have that $\C X_\al$ satisfies the conditions of Assumption \ref{a:faithCES}. Indeed, we may to this end use the complete isometries $\ga_{\C A_\al} : \C A_\al \to B(H)$, $\ga_{\C B_\al} : \C B_\al \to B(H)$ and $\ga_{\C X_\al} : \C X_\al \to B(H)$ defined in \eqref{eq:alpisometI} and \eqref{eq:alpisometII}. The following result is therefore a corollary to Theorem \ref{t:isochar}:

\begin{thm}\label{t:bouchar}
 Let $\C A$ and $\C B$ be operator $*$-algebras and let $\C X$ be an operator $*$-correspondence from $\C A$ to $\C B$. Then there exists a standard form representation $(H_\pi,U,\pi_{\C A}, \pi_{\C B},\pi_{\C X})$ for $\C X$.
\end{thm}

\subsection{The linking algebra of a standard form representation}
Let $\C X$ be an operator $*$-correspondence over a pair $(\C A,\C B)$ of operator $*$-algebras. By Theorem \ref{t:bouchar} we may choose a standard form representation $(H_\pi, U, \pi_{\C A}, \pi_{\C B}, \pi_{\C X})$ of $\C X$. This data then provides us with a completely bounded linear map
\[
\pi_{\C X^{\da}}:\C X^{\da}\to B(H_\pi), \q \pi_{\C X^{\da}}(x^{\da}):=U\pi_{\C X}(x)^{*}U \, .
\]
such that $\C X^\da$ becomes cb-isomorphic (as an operator space) to the image $\pi_{\C X^\da}(\C X^\da) \su B(H_\pi)$. Moreover, we have the relations
\[
\pi_{\C X^\da}(x^\da \cd a) = \pi_{\C X^\da}(x^\da) \cd \pi_{\C A}(a), \q 
\pi_{\C X^\da}(b \cd x^\da) = \pi_{\C B}(b) \cd \pi_{\C X^\da}(x^\da), \q
\pi_{\C X^\da}(x^{\da})\pi_{\C X}(y) = \pi_{\C B}(\langle x,y\rangle) \, .
\]
for all $x, y \in \C X$, $a \in \C A$, $b \in \C B$. In particular, we obtain a completely bounded algebra homomorphism
\[
\pi_{\mathbb{K}(\C X)} : \mathbb{K}(\C X) \to B(H_\pi), \q \pi_{\mathbb{K}(\C X)}(x \ot y^\da) := \pi_{\C X}(x) \cd \pi_{\C X^\da}(y^\da) \, .
\]

In order to ease the notation we will sometimes omit the subscripts from the various maps $\pi_{\C A}$, $\pi_{\C B}$, $\pi_{\C X}$, $\pi_{\C X^\da}$, $\pi_{\mathbb{K}(\C X)}$.

\begin{dfn}\label{d:link} Let $\C A$ and $\C B$ be operator $*$-algebras and let $\C X$ be an operator $*$-correspondence from $\C A$ to $\C B$. For a standard form representation $\pi := (H_\pi, U, \pi_{\C A}, \pi_{\C B}, \pi_{\C X})$ of $\C X$ we define the \emph{linking algebra} $\mathscr{L}_{\pi}(\C X) \su B( H_\pi \op H_\pi)$ as the operator norm closure of the subalgebra
\[
\ma{cc}{ \pi_{\C B}(\C B) & \pi_{\C X}(\C X^\da) \\ \pi_{\C X}(\C X) & \pi_{\C A}(\C A) + \pi_{\mathbb{K}(\C X)}( \mathbb{K}(\C X)) }
\su B(H_\pi \op H_\pi) \, . 
\]

\end{dfn}

Clearly, the linking algebra $\sL_\pi(\C X) \su B(H_\pi \op H_\pi)$ is an operator algebra. Moreover, $\mathscr{L}_{\pi}(\C X)$ admits a $*$-structure turning it into an operator $*$-algebra: 

\begin{lemma} Let $\pi$ be a standard form representation of $\C X$. The involution 
\[
\begin{pmatrix}
\pi(b)& \pi(y^\da) \\ \pi(x) & \pi(a)+ \pi(K) \end{pmatrix}^{\da}
:= \begin{pmatrix} U & 0 \\ 0 & U\end{pmatrix}\begin{pmatrix}\pi(b)& \pi(y^\da) \\ \pi(x) & \pi(a)+\pi(K)\end{pmatrix}^{*} \begin{pmatrix} U & 0 \\ 0 & U\end{pmatrix} \, ,
\]
makes the linking algebra $\mathscr{L}_{\pi}(\C X)$ into an operator $*$-algebra. The triple $\left( H_\pi \op H_\pi, \io, \ma{cc}{U & 0 \\ 0 & U}\right)$, where $\io : \mathscr{L}_{\pi}(\C X) \to B( H_\pi \op H_\pi)$ denotes the inclusion, is a standard form representation of $\sL_{\pi}(\C X)$.
\end{lemma}
\begin{proof} The identity
\[
\begin{pmatrix} U & 0 \\ 0 & U\end{pmatrix}\begin{pmatrix}\pi(b)& \pi(y^\da) \\ \pi(x) & \pi(a)+ \pi(K)\end{pmatrix}^{*} \begin{pmatrix} U & 0 \\ 0 & U\end{pmatrix}=\begin{pmatrix} \pi(b^{\da}) & \pi(x^\da) \\ \pi(y) & \pi(a^{\da})+ \pi(K^\da)\end{pmatrix},
\]
is straightforward to verify and proves the lemma.
\end{proof}

Our results show that one can always construct operator $*$-algebras containing a given operator $*$-correspondence. Conversely, we will now show that corners in a standard form representation of an operator $*$-algebra are always operator $*$-correspondences.

The following result should be compared to \cite[Theorem 5.5]{Ble:Gen} as well as to the (now classical) result on Hilbert $C^{*}$-modules \cite[Theorem 1.1]{BGR}. 
 
\begin{thm} Let $\C L$ be an operator $*$-algebra and let $(H,\pi,U)$ be a standard form representation of $\C L$. Moreover, let $P\in B(H)$ be an orthogonal projection such that
\begin{enumerate}
\item $PU=UP$;
\item $P\pi(\C L) \, , \, \, \pi(\C L)P\su \pi(\C L)$.

\end{enumerate}
Then the closed subalgebras $P\pi(\C L)P$ and $\C B:=(1-P)\pi(\C L)(1-P)$ of $B(H)$ are operator $*$-algebras in the involution $T^{\da}:=UT^{*}U$, $T \in B(H)$. Furthermore, for any closed subalgebra $\C A\su P\pi(\C L)P$ with $\C A^\da = \C A$, the closed subspace $\C X:=P\pi(\C L )(1-P)\su B(H)$ is an operator $*$-correspondence from $\C A$ to $\C B$ with bimodule structure coming from the algebra structure on $B(H)$ and with the inner product
\begin{equation}\label{eq:innercorner}
\inn{S,T}_{\C X} := S^\da \cd T := U S^* U T, \q S, T \in \C X \, .
\end{equation}
Moreover, up to cb-isomorphism, any operator $*$-correspondence arises in this way.
\end{thm}
\begin{proof} 
By construction, the closed subalgebra $\pi(\C L) \su B(H)$ is an operator $*$-algebra in the involution $\pi(x)^\da := U \pi(x)^* U = \pi(x^\da)$, $x \in \C L$, and $\pi : \C L \to \pi(\C L)$ is then a cb-isomorphism of operator $*$-algebras. It is then immediate from condition (1) and (2) that the closed subalgebras $P \pi(\C L) P$ and $\C B = (1 - P) \pi(\C L) (1 - P)$ of $\pi(\C L)$ are operator $*$-algebras when given the involution inherited from $\pi(\C L)$. Hence any closed subalgebra $\C A \su P \pi(\C L) P$ with $\C A = \C A^\da$ is also an operator $*$-algebra.

The operator space $\C X = P \pi(\C L) (1 - P) \su \pi(\C L)$ clearly becomes an operator $\C A$-$\C B$-bimodule when equipped with the bimodule structure inherited from the algebra structure in $\pi(\C L)$. Moreover, the pairing defined in \eqref{eq:innercorner} takes values in $\C B$ since
\[
\begin{split}
\inn{P \pi(x) (1 - P), P \pi(y) (1 - P)}_{\C X} 
& = U ( P \pi(x) (1 - P))^* U P \pi(y) (1 - P)  \\
& = (1 - P) \pi(x^\da) P \pi(y) (1 - P) ,
\end{split}
\]
for all $x,y \in \C X$. It is then clear that $\C X$ is an operator $*$-correspondence from $\C A$ to $\C B$.

Conversely, suppose $\C X$ is an operator $*$-correspondence from $\C A$ to $\C B$ and let $\pi = (H_\pi,V, \pi_{\C A}, \pi_{\C B},\pi_{\C X})$ be a standard form representation of $\C X$, which exists by Theorem \ref{t:bouchar}. The linking algebra $\mathscr{L}_{\pi}(\C X)\su B(H_\pi \oplus H_\pi)$, the unitary and the projection 
\[
U:=\begin{pmatrix} V & 0 \\0 & V \end{pmatrix}\in B(H_\pi \oplus H_\pi ),\q P:=\begin{pmatrix} 0 & 0 \\0 & 1 \end{pmatrix}\in B(H_\pi \oplus H_\pi) \, ,
\]
satisfy conditions (1) and (2). Moreover, $\C X$ is cb-isomorphic (as an operator space) to $P\mathscr{L}_{\pi}(\C X)(1-P)$, $\C B$ is cb-isomorphic to $(1-P)\mathscr{L}_{\pi}(\C X)(1-P)$ (as an operator $*$-algebra) and $\C A$ is cb-isomorphic to a closed $*$-subalgebra of $P\mathscr{L}_{\pi}(\C X) P$ (as an operator $*$-algebra), and these cb-isomorphisms can be chosen to be compatible in the sense described after Definition \ref{d:opstarcor}.
\end{proof}

We note that due to the lack of an appropriate notion of adjointable operators on operator $*$-modules, we are confined, in the above theorem, to a formulation involving standard form representations. A more intrinsic characterization of the linking operator $*$-algebra is the subject of further investigations and requires a better understanding of the (a priori) complicated representation theory of the operator $*$-algebra $\mathbb{K}(\C X)$.

\bibliography{JK}
\bibliographystyle{amsalpha-lmp}

\end{document}